\documentclass[a4paper,12pt]{article}

\usepackage{amsmath,amsthm,amssymb,enumerate}
\usepackage{xcolor}
\usepackage{float}
\usepackage{multirow}
\usepackage{longtable}
\usepackage{diagbox}
\usepackage{soul}
\usepackage{cancel}
\usepackage{gensymb}
\usepackage{enumitem}
\usepackage[square,sort&compress,comma,numbers]{natbib}
\usepackage{bm}
\usepackage{hyperref}
\usepackage{subfig}
\usepackage{graphicx}
\usepackage[margin=1.7cm]{geometry}
\usepackage{tikz}
\usepackage{esint}
\usepackage{caption}
\usetikzlibrary{shapes,calc}
\usepackage{verbatim}
\usepackage{array}
\usepackage{bm}
\definecolor{maroon}{RGB}{144,0,32}
\newcolumntype{H}{>{\setbox0=\hbox\bgroup}c<{\egroup}@{}}

\usepackage{newtxtext,newtxmath}

\usepackage{multirow}
\usepackage{marginnote}
\chardef\bslash=`\\ 

\newtheorem{thm}{Theorem}[section]

\newtheorem{lem}[thm]{Lemma}

\newtheorem{rem}[thm]{Remark}

\theoremstyle{definition}
\theoremstyle{remark}

\numberwithin{equation}{section}

\newcommand{\bN}{\mathbb N}

\newcommand{\bX}{{\bf X}}

\newcommand{\cA}{\mathcal A}

\newcommand{\cE}{\mathcal E}

\newcommand{\cT}{\mathcal{T}}

\newcommand{\err}{{\rm err}}

\newcommand{\hcA}{\widehat{\mathcal A}}

\newcommand{\hto}{H^2_0(\Omega)}

\newcommand{\vket}{von K\'{a}rm\'{a}n equations}
\newcommand{\integ}{\int_\Omega}

\newcommand{\fl}{\quad \text{for all}\:}

\newcommand{\half}{\frac{1}{2}}

\newcommand{\dx}{{\rm\,dx}}
\newcommand{\ds}{{\rm\,ds}}

\newcommand{\cof}{{\rm cof}}
\newcommand{\dg}{{\rm dG}}

\newcommand{\Poincare}{Poincar\'e}
\newcommand{\Holder}{H\"{o}lder~}

{\algorithm}%
{\endalgorithm}

\theoremstyle{definition}

\numberwithin{equation}{section}

\newcommand{\bV}{\text{\bf V}}

\newcommand{\bz}{\boldsymbol{z}}

\newcommand{\bv}{\boldsymbol{v}}

\newcommand{\bW}{\boldsymbol{W}}
\newcommand{\bH}{\boldsymbol{H}}
\newcommand{\bL}{\boldsymbol{L}}

\newcommand{\T}{\mathcal{T}}
\renewcommand{\P}{\mathcal{P}}

\newcommand{\hB}{\widehat{B}}

\def \R{{{\Bbb R}}}
\def \P{{{\mathcal P}}}

\allowdisplaybreaks

\def\R{\mathbb{R}}

\def\cA{\mathcal{A}}

\def\O{\Omega}

\def\bv{{\mathbf V}}

\def\pw{\rm {pw}}

\def\cE{{\mathcal{E}}}

\def\jump#1{\left[\hskip -3.5pt\left[#1\right]\hskip -3.5pt\right]}

\def\bxi{\boldsymbol{\xi}}

\newcommand{\be}{\begin{equation}}
\newcommand{\ee}{\end{equation}}

\usepackage{mathtools}  
\mathtoolsset{showonlyrefs} 
\usepackage[normalem]{ulem}
\definecolor{violet}{rgb}{0.580,0.,0.827}

\usepackage[normalem]{ulem}
\newcounter{corr}
\definecolor{violet}{rgb}{0.580,0.,0.827}
\newcommand{\corr}[3]{\typeout{Warning : a correction remains in page
		\thepage}
	\stepcounter{corr}        
	{\color{red}\ifmmode\text{\,{\ensuremath{#1}}\,}\else{#1}\fi}
	{\color{blue}#2}
	{\color{violet} #3}}

\newcounter{changeto}
\newcommand{\changeto}[2]{\typeout{Warning : a correction remains in page
		\thepage}
	\stepcounter{changeto}        
	{\color{blue}\ifmmode\text{\,\sout{\ensuremath{#1}}\,}\else\sout{#1}\fi}
	{\color{red}#2}}

\title{Morley Type Virtual Element Method for Von K\'{a}rm\'{a}n Equations}
\author{Devika Shylaja, Sarvesh Kumar \footnote{Department of Mathematics, Indian Institute of Space Science and Technology, Thiruvanathapuram 695547, India. devikas.pdf@iist.ac.in, sarvesh@iist.ac.in}}

\begin{document}
	\maketitle
\abstract This paper analyses the nonconforming Morley type virtual element method to approximate a regular solution to the \vket\, that describes bending of very thin elastic plates. Local existence and uniqueness of a discrete solution to the non-linear problem is discussed. A priori error estimate in the energy norm is established under minimal regularity assumptions on the exact solution. Error estimates in piecewise $H^1$ and $L^2$ norm are also derived. A working procedure to find an approximation for the discrete
solution using Newtons method is discussed. Numerical results that justify theoretical estimates are presented.

\section{Introduction}\label{sec:intro}
\noindent 
The \vket\, \cite{CiarletPlates,Knightly,BergerFife} model the bending of very thin elastic plates through a system of fourth-order semi-linear elliptic equations defined by: for a given load $f \in L^2(\Omega)$, seek the vertical displacement $u$ and the Airy stress function $v$ such that
\begin{subequations}\label{vke}
\begin{align}
&	\Delta^2 u =[u,v]+ f \text{ and }\Delta^2 v =-\half[u,u] 
	\text{ in } \Omega, \label{vke.domain}\\
&	u=\frac{\partial u}{\partial n} = v = \frac{\partial v}{\partial n} = 0 \text{  on  } \partial\Omega, \label{vke.bdy}
\end{align}
\end{subequations}\noeqref{vke.domain,vke.bdy}
\noindent with the von K\'{a}rm\'{a}n bracket $\displaystyle
[\eta,\chi]:=\eta_{xx}\chi_{yy}+\eta_{yy}\chi_{xx}-2\eta_{xy}\chi_{xy}$ and $n$ is the unit outward normal to the boundary $\partial \Omega$ of the polygonal domain $\Omega \subseteq \R^2$.

\medskip

\noindent The major challenges of the problem  in its numerical approximation are the non-linearity and the higher order nature of the equations. The results regarding the existence of solutions, regularity and bifurcation phenomena of the \vket\,  in \eqref{vke} are presented in \cite{CiarletPlates, Knightly, Fife, Berger,BergerFife, BlumRannacher} and the references therein. It is well-known \cite{BlumRannacher} that the solutions of the \vket\, belong to $\hto\cap H^{2+\alpha}(\Omega)$, where $\alpha\in (\half,1]$, referred to as the index of elliptic regularity, is determined by the interior angles of $\Omega$. Note that when $\Omega$ is convex, $\alpha=1$.

\medskip

\noindent The numerical methods to approximate the regular solutions of \vket\, has been studied using conforming {finite element methods (FEMs)} in \cite{Brezzi,ng1},  nonconforming Morley FEM in \cite{ng2,carstensen2017nonconforming}, mixed FEMs  in \cite{Miyoshi,Chen2020AMF,Reinhart}, discontinuous Galerkin methods, $C^0$ interior penalty methods in \cite{brennernew,CCGMNN18}, and hybrid FEMs in \cite{Quarteroni}. More recently, a conforming virtual element method is analysed in \cite{Carlo_VEMvKE}. 

\medskip

\noindent The Virtual Element Method (VEM) \cite{Veiga_basicVEM}, which is a generalization of the FEM,  has got more and more attention in recent years, because it can deal with the polygonal meshes and avoid an explicit construction of the discrete shape function, \cite{Veiga_HdivHcurlVEM,Veiga_hitchhikersVEM,Brenner_errorVEM}. The polytopal meshes can be very useful for a wide range of reasons, including meshing of the domain (such as cracks) and data features, automatic use of hanging nodes, adaptivity.  A conforming VEM for plate bending problems is introduced in \cite{Brezzi_VEM_platebending}. A $C^1$ virtual element for the Cahn-Hilliard equations and the vibration problem of Kirchhoff plates is developed in \cite{Antonietti_CahnHilliard}, \cite{}. This has been extended to the \vket\, to approximate the regular solutions in \cite{Carlo_VEMvKE}. In \cite{Zhao_ncfem}, a $C^0$ noncoforming VEM for plate bending problems is constructed for any order of accuracy. This nonconforming method is modified to fully nonconforming Morley type VEM in \cite{Zhao_Morley,Antonietti_ncVEM}. Note that both these papers deal with the same degrees of freedom whereas use different definition on the local virtual space. Recently, the Morley type VEM is analysed for the Navier-Stokes equations in stream function vorticity formulation in \cite{Adak_MorleyVEMNSE}.

\medskip

\noindent The aim of this paper is to extend and analyze the nonconforming Morley type VEM presented in \cite{Zhao_Morley} to approximate a regular solution to  the \vket. Since the discrete space is not a subspace of $H^2_0(\O)$, the convergence analysis offers a lot of challenges and novelty for this semilinear problem with trilinear nonlinearity. The trilinear form in \cite{Adak_MorleyVEMNSE} for the Navier-Stokes equation in  stream-function form vanishes whenever the second and third variables are equal, and satisfies the anti-symmetric property with respect to the second and third variables, and
this aids the wellposedness of the discrete formulation and error analysis. However, the trilinear form for \vket\, does not satisfy the properties stated above and hence leads to interesting challenges in the analysis. A discrete version of Sobolev embedding is employed for establishing the well-posedness of the discrete linearized problem or equivalently a discrete inf-sup condition. This discrete inf-sup condition allows the proof of local existence and uniqueness of a discrete solution to the non-linear problem with a Banach fixed point theorem. Optimal order error estimate in $H^2$ and $H^1$ norms are established using minimal regularity assumption of the exact solution. The discrete non-linear problem can be solved using the Newton’s method by choosing an appropriate initial guess such that there exists a closed sphere in which the approximate solution is unique and the Newton’s iterates converge quadratically to the discrete solution.

\medskip

\noindent The remaining parts are organised as follows. Section~\ref{sec:vke} discusses the weak formulation of the \vket\, and the linearised problem. Section~\ref{sec:MorleyVEM} deals with the Morley type VEM for the \vket. Some auxiliary results required for the convergence analysis and the wellposedness for the discrete linearised problem are established in Section~\ref{sec:wellposed} and is followed by the existence and local uniqueness of the discrete solution using fixed point of a non-linear operator. A priori error control in $H^2$ and $H^1$ norms, and convergence of the Newtons method are derived in Section~\ref{sec:error}. Section~\ref{sec.numericalresults} provides the results of computational experiments that validate the theoretical estimates.

\medskip

\noindent Throughout the paper, standard notations on Lebesgue and Sobolev spaces and their norms are employed. The standard semi-norm and norm on $H^{s}(\Omega)$ (resp. $W^{s,p} (\Omega)$) for $s>0$ and $1 \le p \le \infty$ are denoted by $|\cdot|_{s}$ and $\|\cdot\|_{s}$ (resp. $|\cdot|_{s,p}$ and $\|\cdot\|_{s,p}$ ). The norm in $H^{-s}(\Omega)$ is denoted by $\|\cdot\|_{-s}$. The standard $L^2$ inner product  and norm are denoted by $(\cdot, \cdot)$ and $\|\cdot\|.$ The notation ${\bH}^s(\Omega)$ (resp. ${\boldsymbol{L}}^p(\Omega)$) is used to denote the product space $H^{s}(\Omega) \times H^s(\Omega)$ (resp. $L^p(\Omega) \times L^p(\Omega)$). For all $\Phi = (\varphi_1,\varphi_2) \in {\bH}^s(\Omega) \; ( \text{ resp. } {\boldsymbol L}^2(\Omega))$, the product space is equipped with the norm
$\|{\Phi}\|_{s}:=(\| \varphi_1\|_s^2 +\|\varphi_2\|_s^2)^{1/2} \; 
(\text{ resp. } \|{\Phi}\|:=(\| \varphi_1\|^2 +\|\varphi_2\|^2)^{1/2}). \; 
$ The notation $a\lesssim b$ (resp. $a \gtrsim  b$) means there exists a generic mesh independent constant $C$ such that $a\leq Cb$ (resp. $a\ge Cb$).  

\section{Weak formulation}\label{sec:vke}
\noindent This section deals with the continuous weak formulation and its linearisation of the \vket.

\medskip

\noindent For all  $\eta,\chi, \varphi  \in V:=H^2_0(\O)$,  the weak formulation associated with \eqref{vke} seeks  $u,v\in \: V $ such that, for all $(\varphi_{1},\varphi_{2}) \in \bv =: V \times V$,
\begin{subequations}
\begin{align}
&	a(u,\varphi_1)+ b(u,v,\varphi_1) + b(v,u,\varphi_1) = f(\varphi_1)  \label{vk_weak1}\\
&	a(v,\varphi_2) -b(u,u,\varphi_2)   = 0,  \label{vk_weak2}        
\end{align}
\end{subequations}
with, for all  $\eta,\chi, \varphi  \in V$,
\begin{align}
& a(\eta,\chi):=\int_\Omega D^2\eta:D^2\chi \dx,\quad f(\varphi):=\int_\Omega f\varphi \dx  \mbox{ and } \nonumber \\
& b(\eta,\chi,\varphi):=-\half\integ [\eta,\chi]\varphi\dx =\half\int_\Omega \cof (D^2\eta) \nabla \chi\cdot \nabla \varphi \dx, \label{defn.b}
\end{align}
 where $D^2$ is the Hessian matrix, $:$ denotes the scalar product between the matrices and $\cof(D^2\eta)$ denotes the co-factor matrix of $D^2\eta$. It is known that \eqref{vk_weak1}-\eqref{vk_weak2} possesses at least one solution \cite{Knightly,Brezzi,CiarletPlates}.
 
 \medskip
 
 \noindent The combined vector form for  \eqref{vk_weak1}-\eqref{vk_weak2} seeks $\Psi=(u,v)\in \bv$ such that
\begin{equation}\label{VKE_weak}
{A}(\Psi,\Phi)+B(\Psi,\Psi,\Phi)- {F}(\Phi)=0\fl \Phi\in \bv,
\end{equation}
where, for all $\Xi=(\xi_1,\xi_2),\Theta=(\theta_1,\theta_2)$, and $\Phi=(\varphi_1,\varphi_2)\in  \bv$,
\begin{align*}
&	{A}(\Theta,\Phi):={} a(\theta_1,\varphi_1) + a(\theta_2,\varphi_2), \;
	\\
&B(\Xi,\Theta,\Phi):={} b(\xi_1,\theta_2,\varphi_1)+b(\xi_2,\theta_1,\varphi_1)-b(\xi_1,\theta_1,\varphi_2),\,\text{ and }\;\\
& {F}(\Phi) := (f(\varphi_1),0).
\end{align*}
The trilinear form $b(\bullet,\bullet,\bullet)$ is symmetric in all the three variables and so is $B(\bullet,\bullet,\bullet)$. The  boundedness and ellipticity properties stated below hold \cite{CiarletPlates,ng2}:
\begin{equation}
	{A}(\Theta,\Phi)\leq \|\Theta\|_2 \: \|\Phi\|_2,\: {A}(\Theta,\Theta) \geq \|\Theta\|_2^2, \;\text{ and }\; B(\Xi, \Theta, \Phi) \lesssim \|\Xi\|_2 \: \|\Theta\|_2 \: \|\Phi\|_2. \label{eqn:a_vk1}
\end{equation}
	\begin{lem}[\it{a priori bounds}]{\cite{BlumRannacher,ngr}}\label{lem:aprioriboundcts}\noindent
	\noindent For $f\in H^{-1}(\Omega)$, the solution $\Psi$ of \eqref{VKE_weak} belongs to $ \bv \cap {\bH}^{2 + \alpha} (\Omega)$,  $\alpha \in (1/2,1]$, and  satisfies the $a~priori$ bounds
$	\|\Psi\|_2 \lesssim \|f\|_{-1}$ and $\|\Psi\|_{2+\alpha}   \lesssim  \|f\|^3_{-1}+
	\|f\|_{-1}. $
\end{lem}
\noindent Assume that the solution $\Psi=(u,v)$ is {\it regular} \cite{Brezzi,ng2}. That is, the linearised problem defined by: for given $G=(g_1,g_2) \in \bL^2(\O)$, find $\Theta \in \bv$ such that, for all $\Phi \in \bv$,
\begin{equation}\label{eqn.ctslinearised}
\cA(\Theta,\Phi):=A(\Theta,\Phi)+B(\Psi,\Theta,\Phi)+B(\Theta,\Psi,\Phi) =(G,\Phi)
\end{equation}
is well-posed and satisfies the a priori bounds
\begin{equation}\label{eqn.boundlinearised}
\| \Theta \|_2 \lesssim \|G\| \mbox{ and } \| \Theta\|_{2+\alpha}\lesssim \|G\|,
\end{equation}
where $\alpha$ is the index of elliptic regularity. This is equivalent to an inf-sup condition 
\begin{align}\label{inf-sup}
0<\beta:=\inf_{\substack{\Theta\in \bv\\ |\Theta|_2=1}}\sup_{\substack{\Phi\in \bv\\ |\Phi|_2=1}}\cA(\Theta,\Phi).
\end{align}It is well-known \cite{Knightly} that for sufficiently small $f$, the solution is unique and is a regular solution; but this paper aims at a local approximation of an arbitrary regular solution.

\medskip

\noindent	Let $\Psi$ be a regular solution to \eqref{VKE_Morleyweak}. Then the dual problem defined by: for $Q \in \bH^{-1}(\O),$ find $\bxi \in \bv$ such that
	\begin{equation}\label{eqn.dualcts}
	\cA_{}(\Phi,\bxi)=(Q,\Phi)\quad \forall\, \Phi \in \bv
	\end{equation}
	is well-posed and satisfies the a priori bounds \cite{ng2}:
	\begin{equation}\label{eqn.boundduallinearised}
	\| \bxi \|_2 \lesssim \|Q\|_{-1} \mbox{ and } \| \bxi\|_{2+\alpha}\lesssim \|Q\|_{-1}.
	\end{equation}
\section{Morley type virtual element method}\label{sec:MorleyVEM}
This section deals with the Morley type VEM proposed in \cite{Zhao_Morley} for \eqref{VKE_weak}. The Morley type virtual element is a nonconforming virtual element which has fewer degrees of freedom and not even $C^0$ continuous; it is a simplified version of the $C^0$ continuous nonconforming virtual element presented in \cite{Zhao_ncfem}. 

\medskip

\noindent Let $\T_h$ be a decomposition of $\O$ into non-overlapping simple polygons. Let $\cE_h$ denotes the set of edges $e$ in $\T_h$ and $h_K$ denotes the diameter of the element $K$. Let $h_{\max}$ be the maximum of the diameters of all the elements of the mesh, i.e., $h_{\max}=\max_{K \in \T_h} h_K$. For any $K \in \cT_h$, let $n_K$ denotes its unit outward normal vector along the boundary $\partial K$. The unit normal of an edge $e \in \cE_h$ is denoted by $n_e$, whose orientation is chosen arbitrarily but fixed for internal edges and coinciding with the outward normal of $\O$ for boundary edges. Define the jump $\jump{\varphi}_e:=\varphi|_{K_+}-\varphi|_{K_-}$ and the average $\langle \varphi \rangle _e:=\half\left(\varphi|_{K_+}+\varphi|_{K_-}\right)$ across the interior edge $e$ of $\varphi\in H^1(\cT_h)$ of the adjacent triangles  $K_+$ and $K_-$. Extend the definition of the jump and the average to an edge on boundary by $\jump{\varphi}_e:=\varphi|_e$ and $\langle \varphi\rangle_e:=\varphi|_e$ for $e\in \cE(\partial\Omega)$.  For any vector function, the jump and the average are understood component-wise.
\smallskip

\noindent For a non-negative integer $m$ and $D \subseteq \R^2$, ${\mathcal P}_{m}(D)$ denotes the space of polynomials of degree atmost equal to $m$ on $D$ and
\[\P_m(\cT_h):=\{q \in L^2(\O):q|_K \in \P_m(K) \quad \forall\; K \in \cT_h\}.\]
\noindent For $\Phi =(\varphi_1,\varphi_2) \in W^{m,p}(\T_h)$,  where $ W^{m,p}(\T_h)$ denotes the broken Sobolev space with respect to $\T_h$, $\| \Phi \|_{m,p,h}^2:=|\varphi_1|_{m,p,h}^2+|\varphi_2|_{m,p,h}^2$,  and $|\varphi_i|_{m,p,h}=(\sum_{K\in \mathcal{T}}|\varphi_i|^{p}_{m,p,K})^{1/p}$, $i=1,2$; with $|\cdot|_{m,p,K}$ denoting the usual semi-norm in $W^{m,p}(K)$. When $p=2$, the corresponding norms are denoted by $\|\bullet\|_{m,h}$ and $|\bullet|_{m,h}$. The notation $\bX$ is used to denote the product space $X\times X$.

\medskip

\noindent Assume that there exists a positive real number $C_\T$ such that, for every $K \in \T_h$ \cite{Veiga_basicVEM}:
\begin{itemize}
	\item[({\bf A1})]  $K \in \cT_h$ is star-shaped with respect to every point of a ball of radius $C_{\T}h_K$,
	\item[({\bf A2})] the ratio between the shortest edge and the diameter $h_K$ of $K$ is larger than $C_\cT$.
\end{itemize}
\noindent From \cite{LCJH_errorVEM}, we have that if the mesh $\cT_h$ fulfilling the assumptions ({\bf A1}) and ({\bf A2}), then the mesh also satisfies the following property:

\noindent ({\bf P1}): For each $K \in \T_h$, there exists a virtual triangulation $\cT_h^K$ of $K$ such that $\cT_h^K$ is uniformly shape regular and quasi-uniform. The corresponding mesh size $h_T$ of $\cT_h^K$ is proportional to $h_K$. Every edge of $K$ is a side of a certain triangle in $\cT_h^K$.

\medskip

\noindent For every $K \in \cT_h$, the local shape function space $\widetilde{V}_h(K)$ \cite{Zhao_ncfem,Antonietti_ncVEM} is defined by
\begin{equation}\label{def.vhtildek}
 \widetilde{V}_h(K):=\displaystyle\left\{\varphi \in H^2(K);\Delta^2\varphi =0, \varphi_{|e} \in \P_2(e), \Delta \varphi_{|e} \in \P_{0}(e), e \subseteq \partial K \right\}.
\end{equation}
Obviously, $\P_2(K)\subseteq \widetilde{V}_h(K)$. The degrees of freedom on $\widetilde{V}_h(K)$ are
\begin{itemize}
	\item The values of $\varphi(a_i)$, $\forall$ vertex $a_i$,
	\item The moments $\displaystyle \frac{1}{h_e}\int_e\varphi\ds$, $\forall$ edge $e$,
	\item The moments $\displaystyle \int_e\frac{\partial \varphi}{\partial n_e}\ds$, $\forall$ edge $e$.
\end{itemize}
\noindent For each $K$ and any given $\varphi \in \widetilde{V}_h(K)$, define a projection operator $\Pi^K:\widetilde{V}_h(K) \to \P_2(K)\subseteq \widetilde{V}_h(K)$ as the solution to
\begin{subequations}\label{eqn.projectionPiK}
	\begin{align}
a^K(\Pi^K \varphi,q)&=a^K(\varphi,q) \quad \forall \; q \in P_2(K)\label{eqn.projectionPiK1}
\\
\widehat{\Pi^K(\varphi)}&=\widehat{\varphi}\label{eqn.projectionPiK2}\\
\hspace{-2.2cm}\int_{\partial K}\nabla \Pi^K \varphi \ds&=\int_{\partial K}\nabla \varphi \ds\label{eqn.projectionPiK3}
\end{align}
\end{subequations}
\noindent such that $\Pi^K q=q$ for all $q \in \P_2(K)$ and $\Pi^K$ is computable from the above degrees of freedom. Here, $a^K(\bullet,\bullet)$ is the restriction of the continuous bilinear form $a(\bullet,\bullet)$ on the element $K$ and
\begin{equation}\label{def.hat}
\widehat \varphi=\frac{1}{n}\sum_{i=1}^{n}\varphi(a_i).
\end{equation}
\noindent This $C^0$-nonconforming virtual element is modified to fully nonconforming virtual element in \cite{Zhao_Morley} such that the dimension of the shape function space and the degrees of freedom are reduced. Define the local shape function space $V_h(K)$ on a polygon $K$ by
\begin{equation}\label{def.vhk}
\displaystyle V_h(K):=\left\{\varphi \in \widetilde{V}_h(K):\int_e \Pi^K\varphi\ds=\int_e \varphi\ds \; \forall e \subseteq \partial K\right\}.
\end{equation}
Note that $\widetilde{V}_h(K) \subseteq V_h(K)$ and since $\Pi^Kq=q$ for all $q \in \P_2(K)$, $\P_2(K) \subseteq V_h(K)$. The degrees of freedom on ${V}_h(K)$ are
\begin{itemize}
	\item The values of $\varphi(a_i)$, $\forall$ vertex $a_i$,
	\item The moments $\displaystyle \int_e\frac{\partial \varphi}{\partial n_e}\ds$, $\forall$ edge $e$.
\end{itemize}
Comparing with the degrees of freedom associated with $\widetilde{V}_h(K)$, the zero-order moments of $\varphi$ on edges are removed in the above degrees of freedom. 
\noindent Another fully nonconforming virtual element is presented in \cite{Antonietti_ncVEM} with the same degrees of freedom as above, but with a different local virtual space.
\begin{rem}\label{rem.morley}
	The special case of $V_h(K)$ with $K$ as a triangle together with 6 degrees of freedom leads to $V_h(K)=\P_2(K)$. This shows that, for the (lowest-order) triangular case, the simplified nonconforming virtual element coincides with the Morley nonconforming finite element \cite{Ciarlet} with the same degrees of freedom. Hence, the simplified nonconforming virtual element can be viewed as the extension of the Morley element to polygonal meshes.
\end{rem}

\noindent For every decomposition $\T_h$ of $\O$ into simple polygons $K$, define the global space $V_h$ by 
\begin{align}
V_h:=\Big\{&\varphi_h \in L^2(\O);\,\varphi_{h|K} \in V_h(K) \,\forall\; K \in \cT_h, \varphi_h \text{ is continuous at the internal vertices}\\
&\text{ and vanishes at the boundary vertices, }\int_e\jump{\frac{\partial \varphi_h}{\partial n_e}}\ds=0 \; \forall \;e \in \cE_h\Big\}.\label{defn.vh}
\end{align}
\noindent The global degrees of freedom on ${V}_h$ are
\begin{itemize}
	\item The values of $\varphi_h(a_i)$, $\forall$ internal vertex $a_i$,
	\item The moments $\displaystyle \int_e\frac{\partial \varphi_h}{\partial n_e}\ds$, $\forall$ internal edge $e$.
\end{itemize}
The space $V_h$ is not a subspace of $H^2_0(\O)$ and not even $C^0$ continuous over $\O$; hence the simplified virtual element is fully nonconforming. Moreover,
$\dim(V_h)=N_V+N_E,$ where $N_V$ is the number of internal vertices of $\T_h$ and $N_E$ is the number of internal edges, see \cite{Zhao_Morley} for more details.

\smallskip

\noindent For each element $K \in \T_h$, let $\chi_i$ denote the operator associated with the $i$th degree of freedom, $i=1,\cdots,N^K$. The construction of $V_h$ shows that for every smooth enough function $\varphi$ there exists a unique element, usually known as the interpolant of $\varphi$ restricted to $K$, $\varphi_I^K \in V_h(K)$ such that
\begin{equation}\label{def.phiI}
\chi_i(\varphi-\varphi_I^K)=0, \quad i=1,\cdots,N^K.
\end{equation}
\begin{lem}[Interpolation error]\cite{Brenner, Zhao_Morley}\label{lem.vI}
	For every $K \in \cT_h$ and every $\varphi\in H^s(K)$ with $2 \le s\le 3$, it holds that
	\[\|\varphi-\varphi_I^K\|_{m,K} \lesssim h_K^{s-m}|\varphi|_{s,K},\quad m=0,1,2.\]
\end{lem}
\begin{lem}[Polynomial error]\cite{Brenner,Ciarlet,Zhao_Morley}\label{lem.vpi}
	For every $K \in \cT_h$ and every $\varphi\in H^s(K)$ with $2 \le s\le 3$, there exists a polynomial $\varphi_\pi^K \in \P_2(K)$ such that
	\[\|\varphi-\varphi_\pi^K\|_{m,K} \lesssim h_K^{s-m}|\varphi|_{s,K},\quad m=0,1,2.\]
\noindent Also, $\|\varphi-\varphi_\pi^K\|_{1,4,K} \lesssim h_K^{s-3/2}|\varphi|_{s,K}.$
\end{lem}
\noindent For each polygon $K$, define the discrete local bilinear form on $V_h(K) \times V_h(K)$ by
\begin{equation}\label{defn.ahK}
a_h^K(\varphi_h,\psi_h):=a^K(\Pi^K\varphi_h,\Pi^K \psi_h)+S^K(\varphi_h-\Pi^K\varphi_h,\psi_h-\Pi^K\psi_h),\, \forall\,\varphi_h,\psi_h \in V_h(K),
\end{equation}
where $\Pi^K$ is the projection operator defined in \eqref{eqn.projectionPiK} and $S^K(\bullet,\bullet)$ is a symmetric and positive definite bilinear form satisfying
 \begin{equation}\label{defn.SK}
 c_0 a^K(\varphi_h,\varphi_h)\le S^K(\varphi_h,\varphi_h)\le c_1a^K(\varphi_h,\varphi_h) \quad \forall \;\varphi_h \in \ker(\Pi^K)
 \end{equation}
 for some positive constants $c_0$ and $c_1$ independent of $K$ and $h_K$. It is clear from \eqref{defn.SK} that $S^K$ must scale like $a^K(\bullet,\bullet)$ on the kernel of $\Pi^K$. As in \cite{Brezzi_VEM_platebending}, set
 \[S^K(\varphi,\psi)=\sum_{i=1}^{N^K}\chi_i(\varphi)\chi_i(\psi)h_i^{-2},\]
 where $h_i$ is the characteristic length attached to each degree of freedom $\chi_i$. 
 
 \medskip
 
 \noindent The standard arguments \cite{Veiga_basicVEM} reveals the consistency and stability properties of $a_h^K(\bullet,\bullet)$. That is, 
 \begin{equation}
 a_h^K(p,\varphi_h)=a^K(p,\varphi_h) \quad \forall\, p \in \P_2(K),\,\forall \,\varphi_h \in V_h(K)\label{eqn.consistency}
\end{equation}
and there exists two positive constants $\alpha_{\ast}$ and $\alpha^{\ast}$ independent of $h$ and $K$ such that
\begin{equation}
 \alpha_\ast a^K(\varphi_h,\varphi_h)\le a_h^K(\varphi_h,\varphi_h)\le \alpha^\ast a^K(\varphi_h,\varphi_h)\quad \forall \,\varphi_h \in V_h(K). \label{eqn.stability}
 \end{equation}
 \noindent The global discrete bilinear form $a_h:V_h\times V_h \to \R$ is then defined by
 \begin{equation}\label{defn.ah}
 a_h(\varphi_h,\psi_h):=\sum_{K \in \cT_h}a_h^K(\varphi_h,\psi_h)\quad \forall \; \varphi_h,\psi_h \in V_h.
 \end{equation}
The construction of the discrete trilinear form associated with discrete weak formulation is as follows. Define, for all  $\varphi_h,\psi_h,\theta_h \in V_h(K)$,
\begin{equation}\label{defn.bhk}
b_{h}^K(\varphi_h,\psi_h,\theta_h)=\int_K \cof(D^2(\Pi^K\varphi_h))\nabla (\Pi^K \psi_h)\cdot \nabla (\Pi^K\theta_h)\dx.
\end{equation}
 Then the global discrete trilinear form $b_h:V_h \times V_h \times V_h \to \R$ is defined by
  \begin{equation}\label{defn.bh}
 b_h(\varphi_h,\psi_h,\theta_h):=\sum_{K \in \cT_h}b_h^K(\varphi_h,\psi_h,\theta_h)\quad \forall \; \varphi_h,\psi_h,\theta_h \in V_h.
 \end{equation}
 For constructing the linear form on the right-hand side, let $P_0^K(f)$ denote the $L^2$ projection of load $f$ onto $\P_0(K)$.
 Then the right hand side is defined by \cite{Zhao_Morley}
 \begin{equation}\label{defn.fh}
\displaystyle  \langle f_h,\varphi_h\rangle := 
\sum_{K \in \cT_h}(P_0^K(f),\widehat{\varphi_h})
 \end{equation}
with $\widehat{\varphi_h}$ from \eqref{def.hat}. The Morley type nonconforming virtual element discretisation associated with \eqref{VKE_weak} seeks $\Psi_h:=(u_h,v_h) \in \bv_h:=V_h \times V_h$ such that
\begin{equation}\label{VKE_Morleyweak}
A_h(\Psi_h,\Phi_h)+B_h(\Psi_h,\Psi_h,\Phi_h)=F_h(\Phi_h) \quad \forall\,\Phi_h\in \bv_h,
\end{equation}
where for all $\Xi_h=(\xi_{1,h},\xi_{2,h}),\Theta_h=(\theta_{1,h},\theta_{2,h})$, and $\Phi=(\varphi_{1,h},\varphi_{2,h})\in  \bv_h$,
\begin{subequations}
\begin{equation}
\hspace{-6.4cm}{A}_h(\Theta_h,\Phi_h):={} a_h(\theta_{1,h},\varphi_{1,h}) + a_h(\theta_{2,h},\varphi_{2,h}), \;\label{defn.Ah}
\end{equation}
\vspace{-0.7cm}
\begin{equation}
B_h(\Xi_h,\Theta_h,\Phi_h):={} b_h(\xi_{1,h},\theta_{2,h},\varphi_{1,h})+b_h(\xi_{2,h},\theta_{1,h},\varphi_{1,h})-b_h(\xi_{1,h},\theta_{1,h},\varphi_{2,h}),\label{defn.Bh}
\end{equation}
\vspace{-0.6cm}
\begin{equation}
\hspace{-9.8cm}F_h(\Phi_h)=(\langle f_h,\varphi_{1,h}\rangle,0)\label{defn.Fh}
\end{equation}
\end{subequations}
with $a_h(\bullet,\bullet)$, $b_h(\bullet,\bullet,\bullet)$ and $\langle f_h,\bullet\rangle$ from \eqref{defn.ah}, \eqref{defn.bh}, and \eqref{defn.fh} respectively.
\medskip

\noindent For all $\varphi+\varphi_h,\psi+\psi_h, \theta+\theta_h \in V+V_h$, extend the definition of $b_h^K(\bullet,\bullet,\bullet)$ on $V_h(K)$ to $\widehat{b}_h^K(\bullet,\bullet,\bullet)$ on $V+V_h(K)$ as
\begin{equation}\label{defn.bhkhat}
\widehat{b}_h^K(\varphi+\varphi_h, \psi+\psi_h,\theta+\theta_h)=\int_K \cof(D^2(\varphi+\Pi^K\varphi_h))\nabla (\psi+\Pi^K \psi_h)\cdot \nabla (\theta+\Pi^K\theta_h)\dx.
\end{equation}
 Then the global discrete trilinear form $\widehat{b}_h:(V+V_h) \times (V+V_h)\times (V+V_h)\to \R$ is defined by
\begin{equation}\label{defn.bhhat}
\widehat{b}_h(\varphi+\varphi_h, \psi+\psi_h,\theta+\theta_h):=\sum_{K \in \cT_h}\widehat{b}_h^K(\varphi+\varphi_h, \psi+\psi_h,\theta+\theta_h).
\end{equation}
For all $\Xi_h=(\xi_{1,h},\xi_{2,h}),\Theta_h=(\theta_{1,h},\theta_{2,h})$, and $\Phi=(\varphi_{1,h},\varphi_{2,h})\in  \bv+\bv_h$,
\begin{equation}\label{defn.Bhhat}
\hB_h(\Xi_h,\Theta_h,\Phi_h):={} \widehat{b}_h(\xi_{1,h},\theta_{2,h},\varphi_{1,h})+\widehat{b}_h(\xi_{2,h},\theta_{1,h},\varphi_{1,h})-\widehat{b}_h(\xi_{1,h},\theta_{1,h},\varphi_{2,h}).
\end{equation}
\noindent The Morley type nonconforming virtual element formulation corresponding to the continuous linearised problem \eqref{eqn.ctslinearised} seeks $\Theta_h \in \bV_h$ such that
\begin{equation}\label{eqn.discretelinearised}
\cA_{h}(\Theta_h,\Phi_h):=A_h(\Theta_h,\Phi_h)+\hB_h(\Psi,\Theta_h,\Phi_h)+\hB_h(\Theta_h,\Psi,\Phi_h) =G_h(\Phi_h)\quad \forall \, \Phi_h \in \bV_h,
\end{equation}
where $G_h(\Phi_h)=(\langle g_{1,h},\varphi_{1,h}\rangle,\langle g_{2,h},\varphi_{2,h} \rangle)$ with $\langle \bullet,\bullet\rangle$ as in \eqref{defn.fh}.

\medskip

\noindent Define the broken semi-norm on $V_h$ by
\begin{equation}\label{eqn.norm}
|\varphi_h|_{m,h}:=\left(\sum_{K \in \cT_h}|\varphi_h|_{m,K}^2\right)^{1/2}, \quad m=1,2.
\end{equation}
Then, the piecewise version of the energy norm in $H^2(\cT_h)\equiv \prod\limits_{K \in \cT} H^2(K)$, $|\bullet|_{2,h}$, is a norm on $V_h$ \cite[Lemma 5.1]{Zhao_Morley}. This, in particular, implies
\begin{equation}\label{eqn.seminormbound}
\|\varphi_h\|^2+|\varphi_h|_{1,h}^2 \lesssim |\varphi_h|_{2,h}^2.
\end{equation}

%
\section{Well-posedness, existence and uniqueness}\label{sec:wellposed}
\noindent This section presents some auxiliary results that are useful to establish the convergence analysis and is followed by the well-posedness of the discrete linearised problem in Section \ref{sec:wellposedness_subsec}. Section~\ref{sec:existence} discusses the existence and local uniqueness of the discrete solution.
\begin{lem}[Boundedness and coercivity]\label{lem:Ahproperties}
Any $\Phi_h,\Theta_h \in \bv_h$ satisfy
\begin{itemize}
	\item[(a)] $A_h(\Phi_h,\Theta_h)\lesssim |\Phi_h|_{2,h}|\Theta_h|_{2,h}$.
	\item[(b)] $A_h(\Phi_h,\Phi_h)\gtrsim |\Phi_h|_{2,h}^2.$
	\item[(c)] $F_h(\Phi_h) \lesssim \|f\||\Phi_h|_{2,h}.$
\end{itemize}	
\end{lem}
\noindent	{\it Proof of (a).} The symmetry of $a_h(\bullet,\bullet)$, \eqref{eqn.stability}, and the definition of $a^K(\bullet,\bullet)$ imply, for all $\varphi_h,\theta_h \in V_h$,
	\begin{align}
	a_h^K(\varphi_h,\theta_h)\le (a_h^K(\varphi_h,\varphi_h))^{1/2} (a_h^K(\theta_h,\theta_h))^{1/2}&\le \alpha^\ast(a^K(\varphi_h,\varphi_h))^{1/2} (a^K(\theta_h,\theta_h))^{1/2}\\
	&=\alpha^\ast|\varphi_h|_{2,K}|\theta_h|_{2,K}.
	\end{align}
	The sum over all $K \in \T_h$, \Holder inequality, the definition of $|\bullet|_{2,h}$ in \eqref{eqn.norm}, and \eqref{defn.Ah} conclude the proof of $(a)$. \qed
	
	\smallskip
	
\noindent {\it Proof of (b).} The estimate follows from the definition of $A_h(\bullet,\bullet)$ in \eqref{defn.Ah}, the stability property \eqref{eqn.stability}, the definition of $a^K(\bullet,\bullet)$, and \eqref{eqn.norm}.\qed

\smallskip

\noindent {\it Proof of (c).} The definition of $F_h(\bullet)$ in \eqref{defn.Fh}, \eqref{defn.fh}, \Holder inequality, $\|P_0^K f\|_{0,K} \lesssim \|f\|_{0,K}$, $\|\widehat{\varphi_h}-\varphi_h\|_{0,K} \lesssim h_K \|\varphi_h\|_{1,K}$ for $\varphi_h \in V_h$, and \eqref{eqn.seminormbound} concludes the proof of $(c)$. \qed

\medskip

\noindent Since the discrete space $V_h$ is not a subspace of $V$, an {\it enrichment operator ${E}_h$} which maps the nonconforming discrete space to the conforming space plays an important role in establishing the boundedness properties of the discrete trilinear form and hence to derive the local existence and uniqueness of the discrete solution, and a priori error estimates for the solution of \vket. For that, consider the local finite dimensional space \cite[Section 2.2]{Antonietti_CahnHilliard}:
\begin{align}
\widehat{V}_h^c(K):=\Big\{\varphi_h \in H^2(K);\,&\Delta^2\varphi_h \in \P_2(K),\,\varphi_{h}|_{\partial K}\in C^0(\partial K),\, \varphi_{h}|_e \in \P_3(e)\;\forall e \subseteq \partial K,\\
&\nabla \varphi_{h|\partial K}\in C^0(\partial K)^2,\, \frac{\partial \varphi_h}{\partial n_e}\big|_e \in \P_1(e)\;\forall e \subseteq \partial K\Big\}.
\end{align}
For $K \in \T_h$, the $H^2$ conforming local virtual finite element space $V_h^c(K)$  is then defined by
\begin{align}\label{defn.vhck}
V_h^c(K):=\Big\{\varphi_h \in \widehat{V}_h^c(K);\,&(\varphi_h-\Pi^{K,c}\varphi_h,q)_{0,K}=0\,\forall\, q \in \P_2(K)\Big\},
\end{align}
where $\Pi^{K,c}:\widehat{V}_h^c(K) \to \P_2(K)\subseteq \widehat{V}_h^c(K)$ is the projection operator associated with the conforming virtual element, see \cite[Section 2.2]{Antonietti_CahnHilliard} for more details. The global $C^1$ virtual element space is 
\begin{align}\label{defn.vhc}
V_h^c:=\Big\{\varphi_h \in V; \,\varphi_h|_K \in {V}_h^c(K)\; \forall\; K \in \cT_h\Big\}.
\end{align}
\noindent Let ${E}_h:V_h \to V_h^c$ be the enrichment operator. The properties of ${E}_h$ \cite[Propostion 4.1]{Adak_MorleyVEMNSE}, \cite[Lemma 4.2]{JHYY_mediusVEM} that are useful in the analysis are stated in the next lemma.
\begin{lem}[Enrichment operator]\label{lem:Eh}
For all ${\varphi}_h\in V_h$, ${E}_h {\varphi}_h \in V_h^c$ satisfies
\[\sum_{m=0}^{2}h_K^{m-2}|{\varphi}_h-{E}_h{\varphi}_h|_{m,K} \lesssim |{\varphi}_h|_{2,h}.\]
\end{lem}
\begin{lem}[Discrete Sobolev embeddings] \cite[Theorem 4.1]{Adak_MorleyVEMNSE}\label{lem:discreteSobolev}
	For $2\le q <\infty$, any ${\varphi}_h \in V_h$ satisfies $|{\varphi}_h|_{1,q,h}\lesssim |{\varphi}_h|_{2,h}. $
\end{lem}
\noindent Recall the definition of $\Pi^K$ from \eqref{eqn.projectionPiK}. Define $\Pi^h$ in $L^2(\O)$ as, for all $v \in H^2(\cT_h)$,
\begin{equation}\label{eqn.Pih}
(\Pi^h v)|_K:=\Pi^K v,\quad \forall\;K \in \cT_h.
\end{equation}
\noindent For $\varphi \in V_h$, a choice of $q=\Pi^K\varphi \in \P_2(K)$ in \eqref{eqn.projectionPiK1} leads to $a^K(\Pi^K \varphi,\Pi^K \varphi)=a^K(\varphi,\Pi^K \varphi)$. The definition of $a^K(\bullet,\bullet)$ and Cauchy inequality imply $|\Pi^K \varphi|_{2,K} \le |\varphi|_{2,K}.$ Consequently,
\begin{equation}\label{eqn.Pihbound}
|\Pi^h \varphi|_{2,h} \le |\varphi|_{2,h}.
\end{equation}
\begin{lem}[\Poincare-Freidrich inequality and inverse estimates]\cite{Brenner_PFH2},\cite[Lemma 3.2]{JHYY_mediusVEM},\cite[(2.8)]{Brenner_errorVEM} \label{lem:PF}
Any $\varphi\in H^2(K)$ satisfies
	\begin{itemize}
		\item[(a)]$\displaystyle h_K^{-2}\|\varphi\|_{0,K}\lesssim |\varphi|_{2,K}+h_K^{-2}\left|\int_{\partial K}\varphi\ds\right|+h_K^{-1}\left|\int_{\partial K}\nabla \varphi \ds\right|$,
	\item[(b)]$|\varphi|_{1,K} \lesssim h_K|\varphi|_{2,K}+h_K^{-1}\|\varphi\|_{0,K},$	
	\item[(c)]$\|\varphi\|_{0,\infty,K} \lesssim h_K^{-1}\|\varphi\|_{0,K}+|\varphi|_{1,K}+h_K|\varphi|_{2,K}.$
\end{itemize}
\end{lem}
\begin{lem}[Projection error]\label{lem:vh-PiKvh}
Any $K \in \cT_h$ and $\varphi_h \in V_h$ satisfy
\noindent	\[\|\varphi_h-\Pi^K \varphi_h\|_{0,K}+h_K|\varphi_h-\Pi^K \varphi_h|_{1,K} \lesssim h_K^{2}|\varphi_h-\Pi^K \varphi_h|_{2,K}.\]
\end{lem}
\begin{proof}
\noindent Since $\varphi_h \in V_h$, the choice $\varphi=\varphi_h-\Pi^K \varphi_h$ in Lemma~\ref{lem:PF}.a, \eqref{def.vhk}, and \eqref{eqn.projectionPiK3} lead to $h_K^{-2}\|\varphi_h-\Pi^K \varphi_h\|_{0,K} \lesssim | \varphi_h-\Pi^K \varphi_h|_{2,K}$. A combination of this and Lemma~\ref{lem:PF}.b shows $|\varphi_h-\Pi^K \varphi_h|_{1,K} \lesssim h_K|\varphi_h-\Pi^K \varphi_h|_{2,K}$.
\end{proof}
\noindent 
Recall the definition of $B_h(\bullet,\bullet,\bullet)$ and $\hB_h(\bullet,\bullet,\bullet)$ from \eqref{defn.Bh} and \eqref{defn.Bhhat}, and the index of elliptic regularity $\alpha \in (\half,1]$.
\begin{lem}[Bounds for $\hB_h(\bullet,\bullet,\bullet)$]\label{lem:boundednessBh}Any $\bxi \in \bv\cap \bH^{2+\alpha}(\O)$, $\Phi,\Theta \in \bv$, and $\Phi_h,\Theta_h, \bxi_h \in \bv_h$ satisfy
	\begin{itemize}
		\item[(a)]$B_h(\Phi_h,\Theta_h,\bxi_h)\lesssim |\Phi_h|_{2,h}|\Theta_h|_{2,h}|\bxi_h|_{2,h},$
					\item[(b)] $\hB_h(\Theta+\Theta_h,\Phi+\Phi_h,\bxi+\bxi_h) \lesssim |\Theta+\Pi^h\Theta_h|_{2,h}|\Phi+\Pi^h\Phi_h|_{1,4,h}|\bxi+\Pi^h\bxi_h|_{1,4,h},$	
		\item[(c)] $\hB_h(\bxi,\Theta,\Phi+\Phi_h) \lesssim \|\bxi\|_{2+\alpha}|\Theta|_{2}|\Phi+\Pi^h\Phi|_{1,h},$
		\item[(d)] $\hB_h(\bxi,\Theta_h,\Phi+\Phi_h) \lesssim \|\bxi\|_{2+\alpha}|\Theta_h|_{2,h}|\Phi+\Pi^h\Phi|_{1,h},$
		\item[(e)] $\hB_h(\bxi,\Phi+\Phi_h,\Theta+\Theta_h) \lesssim \|\bxi\|_{2+\alpha}|\Phi+\Pi^h\Phi_h|_{1,h}(|\Theta|_{2}+|\Theta_h|_{2,h}),$
		\item[(f)] $\hB_h(\Theta+\Theta_h,\bxi,\Phi+\Phi_h) \lesssim |\Theta+\Pi^h\Theta_h|_{2,h}\|\bxi\|_{2+\alpha}|\Phi+\Pi^h\Phi_h|_{1,h}.$		
	\end{itemize}
\end{lem}
\noindent	{\it Proof of (a).} For $\varphi_h,\theta_h,\xi_h \in V_h$, the definition of $b_h(\bullet,\bullet,\bullet)$ in \eqref{defn.bh} and a \Holder inequality show
	\[b_h(\varphi_h,\theta_h,\xi_h)\lesssim |\Pi^h \varphi_h|_{2,h}|\Pi^h \theta_h|_{1,4,h}|\Pi^h \xi_h|_{1,4,h}.\]
\noindent Lemma~\ref{lem:discreteSobolev} for $q=4$ and \eqref{eqn.Pihbound} read $b_h(\varphi_h,\theta_h,\xi_h)\lesssim | \varphi_h|_{2,h}| \theta_h|_{2,h}|\xi_h|_{2,h}$. This and the definition of $B_h(\bullet,\bullet,\bullet)$ in \eqref{defn.Bh} concludes the proof of $(a)$. \qed

\medskip

\noindent {\it Proof of (b).} The result follows from the definition of $\widehat{b}_h(\bullet,\bullet,\bullet)$ in \eqref{defn.bhhat}, a \Holder inequality (same as in $(a)$), and \eqref{defn.Bhhat}. \qed

\medskip

\noindent {\it Proof of (c).} The definition of $\widehat{b}_h(\bullet,\bullet,\bullet)$ in \eqref{defn.bhhat} and a \Holder inequality lead to, for all $\xi \in V\cap H^{2+\alpha}(\O)$, $\varphi,\theta \in V$ and $\varphi_h \in V_h$,
\begin{equation}
\widehat{b}_h(\xi,\theta,\varphi+\varphi_h)\lesssim |\xi|_{2,4}|\theta|_{1,4}|\varphi+\Pi^h\varphi_h|_{1,h}.
\end{equation}
The Sobolev embeddings $H^{2+\alpha}(\O) \hookrightarrow W^{2,4}(\O)$ for $\alpha>1/2$ and $H^2(\O) \hookrightarrow W^{1,4}(\O)$ \cite{brezis}, and \eqref{defn.Bhhat} concludes the proof of $(c)$. \qed

\medskip

\noindent {\it Proof of (d).} A \Holder inequality reveals, for all $\xi \in V\cap H^{2+\alpha}(\O)$, $\varphi \in V$ and $\theta_h,\varphi_h \in V_h$,
\begin{equation}
\widehat{b}_h(\xi,\theta_h,\varphi+\varphi_h)\lesssim |\xi|_{2,4}|\Pi^h\theta_h|_{1,4,h}|\varphi+\Pi^h\varphi_h|_{1,h}.
\end{equation}
The Sobolev embedding $H^{2+\alpha}(\O) \hookrightarrow W^{2,4}(\O)$, Lemma~\ref{lem:discreteSobolev}, \eqref{eqn.Pihbound}, and \eqref{defn.Bhhat} proves the assertion of $(d)$. \qed

\medskip

\noindent {\it Proof of (e).} The \Holder inequality, the Sobolev embeddings $H^{2+\alpha}(\O) \hookrightarrow W^{2,4}(\O)$, $H^2(\O) \hookrightarrow W^{1,4}(\O)$, and Lemma~\ref{lem:discreteSobolev} provide
\begin{align}
\widehat{b}_h(\xi,\varphi+\varphi_h,\theta+\theta_h)&\lesssim |\xi|_{2,4}|\varphi+\Pi^h\varphi_h|_{1,h}(|\theta|_{1,4}+|\Pi^h\theta_h|_{1,4,h})\\
&\lesssim \|\xi\|_{2+\alpha}|\varphi+\Pi^h\varphi_h|_{1,h}(|\theta|_{2}+|\Pi^h\theta_h|_{2,h}).
\end{align}
This, \eqref{eqn.Pihbound}, and \eqref{defn.Bhhat} concludes the proof of $(e)$.\qed
\medskip

\noindent {\it Proof of (f).} The \Holder inequality and the Sobolev embedding $H^{2+\alpha}(\O) \hookrightarrow W^{1,\infty}(\O)$ \cite{brezis} show
\begin{align}\label{eqn.bhhatd}
\widehat{b}_h(\theta+\theta_h,\xi,\varphi+\varphi_h)&\lesssim |\theta+\Pi^h\theta_h|_{2,h}|\xi|_{1,\infty}|\varphi+\Pi^h\varphi_h|_{1,h}\nonumber\\
&\lesssim |\theta+\Pi^h\theta_h|_{2,h}\|\xi\|_{2+\alpha}|\varphi+\Pi^h\varphi_h|_{1,h}.
\end{align}
A combination of \eqref{eqn.bhhatd} and \eqref{defn.Bhhat} implies the assertion. \qed
\medskip

\noindent Define, for all $\Theta_h=(\theta_{1,h},\theta_{2,h})$, and $\Phi=(\varphi_{1,h},\varphi_{2,h})\in  \bv+\bv_h$,
\begin{equation}\label{defn.Apw}
A_{\pw}(\Theta_h,\Phi_h):=\sum_{K \in \cT_h} (a^K(\theta_{1,h},\varphi_{1,h}) + a^K(\theta_{2,h},\varphi_{2,h})).
\end{equation}
\noindent Define $\varphi_\pi^h \in \P_2(\cT_h)$ and $\varphi_I^h \in V_h$ by $$\varphi_\pi^h|_K :=\varphi_\pi^K \quad \mbox{ and }\quad \varphi_I^h|_K :=\varphi_I^K$$ for all $K \in \cT_h$, where $\varphi_\pi^K$ and $\varphi_I^K$ are as in Lemma~\ref{lem.vpi} and~\ref{lem.vI}.
\begin{lem}[Bounds for $A_{\pw}(\bullet,\bullet)$]\cite{Adak_Morleywinddriven},\cite[Lemmas 4.2, 4.3]{ScbSungZhang} \label{lem:Apw}
Any $\bxi \in \bH^{2+\alpha}(\O)$ for $\alpha \in (\half,1]$, $\Phi \in \bv \cap  \bH^{2+\alpha}(\O)$, and $\Phi_h \in \bv_h$ satisfy
	\begin{itemize}
		\item[(a)]
$A_{\pw}(\bxi,E_h\Phi_h-\Phi_h)\lesssim h_{\max}^{\alpha}\|\bxi\|_{2+\alpha}|\Phi_h|_{2,h},$
\item[(b)]
$A_{\pw}(\bxi,\Phi-\Phi_I^h)\lesssim h_{\max}^{2\alpha}\|\bxi\|_{2+\alpha}\|\Phi\|_{2+\alpha}$, where $\Phi_I^h \in \bV_h$ is the interpolant of $\Phi$.
	\end{itemize}
\end{lem}
\begin{rem}[consequences of Lemma~\ref{lem.vI},~\ref{lem.vpi}, and~\ref{lem:Eh}]\label{remark.consequence}
	The estimates in Lemma~\ref{lem.vI},~\ref{lem.vpi}, and~\ref{lem:Eh}  give rise some typical estimates utilised throughout the analysis in this paper. For $\varphi \in V \cap H^{2+\alpha}(\O)$, a triangle inequality with $\varphi$, Lemma~\ref{lem.vI}, and Lemma~\ref{lem.vpi} show
	\begin{equation}\label{eqn.vIvpi}
	\|\varphi_I^h-\varphi_\pi^h\|_{2,h} \le 	\|\varphi-\varphi_I^h\|_{2,h}+	\|\varphi-\varphi_\pi^h\|_{2,h} \lesssim h_{\max}^{\alpha}\|\varphi\|_{2+\alpha}.
	\end{equation}
	For $\varphi_h \in V_h$, triangle inequality with $\varphi_h$, Lemma~\ref{lem:Eh},~\ref{lem:vh-PiKvh}, and \eqref{eqn.Pihbound} provide
	\begin{equation}\label{eqn.Eh}
	|E_h\varphi_h-\Pi^h\varphi_h|_{1,h} \le  	|E_h\varphi_h-\varphi_h|_{1,h}+	|\varphi_h-\Pi^h\varphi_h|_{1,h}\lesssim h_{\max}|\varphi_h|_{2,h}.
	\end{equation}
Analog arguments lead to $	|E_h\varphi_h-\Pi^h\varphi_h|_{2,h}\lesssim |\varphi_h|_{2,h}$. Lemma~\ref{lem:vh-PiKvh}, $\Pi^h\varphi_\pi^h=\varphi_\pi^h$, and \eqref{eqn.Pihbound} read
\begin{align}
|\varphi_I^h-\Pi^h\varphi_I^h |_{1,h} &\lesssim h_{\max}|\varphi_I^h-\Pi^h\varphi_I^h |_{2,h} \le h_{\max}(|\varphi_I^h-\varphi_\pi^h|_{2,h}+|\Pi^h(\varphi_\pi^h-\varphi_I^h) |_{2,h})\\
&\lesssim h_{\max}|\varphi_I^h-\varphi_\pi^h |_{2,h}\lesssim h_{\max}^{1+\alpha}\|\varphi\|_{2+\alpha}\label{t3}
\end{align}
with \eqref{eqn.vIvpi} in the last step. This and Lemma~\ref{lem.vI} lead to
\begin{equation}\label{eqn.xiPixiIh1h}
|\varphi-\Pi^h\varphi_I^h |_{1,h}\lesssim  h_{\max}^{1+\alpha}\|\varphi\|_{2+\alpha}\; \mbox{ and }\; |\varphi-\Pi^h\varphi_I^h |_{2,h}\lesssim  h_{\max}^{\alpha}\|\varphi\|_{2+\alpha}.
\end{equation}
\end{rem}
\subsection{Well-posedness of the discrete problem}\label{sec:wellposedness_subsec}
\noindent This section establishes the well-posedness of the discrete linearised problem \eqref{eqn.discretelinearised}. 
\begin{thm}[Well-posedness of discrete linearised problem]\label{thm.discretelinearised}
	Let $\Psi \in \bv \cap \bH^{2+\alpha}(\O)$ be a regular solution to \eqref{VKE_Morleyweak}. Then for sufficiently small $h_{\max}$, the discrete linearised problem \eqref{eqn.discretelinearised} is well-posed.
\end{thm}
\begin{proof}
\noindent Since $\bv_h$ is finite dimensional and \eqref{eqn.discretelinearised} is linear, establishing an a priori bound is sufficient to prove that \eqref{eqn.discretelinearised} has a unique solution. The choice $\Phi_h =\Theta_h$ in \eqref{eqn.discretelinearised}, Lemma~\ref{lem:Ahproperties}.b,~\ref{lem:boundednessBh}.d and .f, and \eqref{eqn.Pihbound} lead to
\begin{align}
G_h(\Theta_h)&=\cA_{h}(\Theta_h,\Theta_h)=A_h(\Theta_h,\Theta_h)+\hB_h(\Psi,\Theta_h,\Theta_h)+\hB_h(\Theta_h,\Psi,\Theta_h)\nonumber\\
&\gtrsim |\Theta_h|_{2,h}^2-\|\Psi\|_{2+\alpha}|\Theta_h|_{2,h}|\Pi^h\Theta_h|_{1,h}.\label{eqn.gh}
\end{align}
The arguments in the proof of Lemma~\ref{lem:Ahproperties}.c show $G_h(\Theta_h) \lesssim \|G\||\Theta_h|_{2,h}$. 
This with the above displayed inequality results in
\begin{equation}\label{eqn.thetah}
|\Theta_h|_{2,h}\lesssim \|\Psi\|_{2+\alpha}|\Pi^h\Theta_h|_{1,h}+\|G\|.
\end{equation}
The triangle inequality and \eqref{eqn.Eh} provide
\begin{equation}\label{eqn.tri.thetah}
|\Pi^h\Theta_h|_{1,h}\le |\Pi^h\Theta_h-E_h\Theta_h|_{1,h}+|E_h\Theta_h|_{1}\lesssim h_{\max}|\Theta_h|_{2,h}+|E_h\Theta_h|_{1}.
\end{equation}
To estimate $|E_h\Theta_h|_{1},$ choose $Q=-\Delta E_h\Theta_h$ and $\Phi=E_h\Theta_h$ in the dual problem \eqref{eqn.dualcts}. The definition of $\cA(\bullet,\bullet)$ in \eqref{eqn.ctslinearised} and \eqref{eqn.discretelinearised} for $\Phi_h=\bxi_I^h$ read
\begin{align}
|E_h\Theta_h|_{1}^2
&=A(E_h\Theta_h,\bxi)+B(\Psi,E_h\Theta_h,\bxi)+B(E_h\Theta_h,\Psi,\bxi)\\
&=A_{\pw}(E_h\Theta_h,\bxi)+\hB_h(\Psi,E_h\Theta_h,\bxi-\bxi_I^h)+\hB_h(E_h\Theta_h,\Psi,\bxi-\bxi_I^h)\nonumber\\
&\quad+\hB_h(\Psi,E_h\Theta_h,\bxi_I^h)+\hB_h(E_h\Theta_h,\Psi,\bxi_I^h)-A_h(\Theta_h,\bxi_I^h)\nonumber\\
&\quad -\hB_h(\Psi,\Theta_h,\bxi_I^h)-\hB_h(\Theta_h,\Psi,\bxi_I^h)+G_h(\bxi_I^h).
\end{align}
Since $\bxi_\pi^h \in \P_2(\cT_h)$, the consistency property in \eqref{eqn.consistency} shows $A_h(\Theta_h,\bxi_\pi^h)=A_{\pw}(\Theta_h,\bxi_\pi^h).$ This and elementary algebra reveal 
\begin{align}
|E_h\Theta_h|_{1,2,h}^2
&=A_{\pw}(E_h\Theta_h-\Theta_h,\bxi)+A_{\pw}(\Theta_h,\bxi-\bxi_\pi^h)+A_h(\Theta_h,\bxi_\pi^h-\bxi_I^h)\nonumber\\
&\quad+\hB_h(\Psi,E_h\Theta_h,\bxi-\bxi_I^h)+\hB_h(E_h\Theta_h,\Psi,\bxi-\bxi_I^h)\nonumber\\
&\quad+\hB_h(\Psi,E_h\Theta_h-\Theta_h,\bxi_I^h)+\hB_h(E_h\Theta_h-\Theta_h,\Psi,\bxi_I^h-\bxi)\nonumber\\
&\quad+\hB_h(E_h\Theta_h-\Theta_h,\Psi,\bxi) +G_h(\bxi_I^h):=T_1+\cdots+T_9.\label{eqn.t1tot9}
\end{align}
\noindent Lemma~\ref{lem:Apw} leads to $T_1 \lesssim h_{\max}^\alpha \|\bxi\|_{2+\alpha}|\Theta_h|_{2,h}$. The Cauchy inequality and Lemma~\ref{lem.vpi} show $T_2\lesssim h_{\max}^\alpha \|\bxi\|_{2+\alpha}|\Theta_h|_{2,h}$. Lemma~\ref{lem:Ahproperties}.a, and \eqref{eqn.vIvpi} provide $T_3\lesssim h_{\max}^\alpha \|\bxi\|_{2+\alpha}|\Theta_h|_{2,h}$. Lemma~\ref{lem:boundednessBh}.c, triangle inequality with $\Theta_h$, Lemma~\ref{lem:Eh} and \eqref{eqn.xiPixiIh1h} read
 $T_4 \lesssim h_{\max}^{1+\alpha}\|\bxi\|_{2+\alpha}|\Theta_h|_{2,h}$. Lemma~\ref{lem:boundednessBh}.f shows
$T_5 \lesssim |E_h \Theta_h|_{2,h}\|\Psi\|_{2+\alpha}|\bxi-\Pi^h\bxi_I^h|_{1,h}$. Lemma~\ref{lem:Eh} and \eqref{eqn.xiPixiIh1h} provide $T_5 \lesssim h_{\max}^{1+\alpha}\|\bxi\|_{2+\alpha}|\Theta_h|_{2,h}$. Lemma~\ref{lem:boundednessBh}.e, \eqref{eqn.Eh}, and Lemma~\ref{lem.vI} reveal  $T_6 \lesssim h_{\max}\|\bxi\|_{2+\alpha}|\Theta_h|_{2,h}$. Lemma~\ref{lem:boundednessBh}.f, Remark~\ref{remark.consequence}, and \eqref{eqn.xiPixiIh1h} lead to $T_7 \lesssim h_{\max}^{1+\alpha}\|\bxi\|_{2+\alpha}|\Theta_h|_{2,h}$. 

\noindent Since $\Psi,\bxi \in \bV \cap \bH^{2+\alpha}(\O)$, the Sobolev embedding $H^{2+\alpha}(\O) \hookrightarrow W^{2,4}(\O)$ shows
\[|\nabla \Psi \cdot \nabla \bxi|_{1} \le |\nabla \Psi|_{1,4}|\nabla \bxi|_{1,4}\lesssim \|\Psi\|_{2+\alpha}\|\bxi\|_{2+\alpha}.\]
Hence, $\nabla \Psi \cdot \nabla \bxi \in H^1(\O)$ and $T_8 \lesssim \sum_{K \in \cT_h}|D^2(E_h \Theta_h- \Pi^h\Theta_h)|_{-1,K}|\nabla \Psi \cdot \nabla \bxi|_{1,K}$. The definition of the dual norm and integration by parts lead to $|D^2(E_h \Theta_h- \Pi^h\Theta_h)|_{-1,K}\lesssim |E_h \Theta_h-\Pi^h\Theta_h|_{1,K}$. The combination of the above estimates and \Holder inequality provides $T_8 \lesssim |E_h \Theta_h- \Pi^h\Theta_h|_{1,h}\|\Psi\|_{2+\alpha}\|\bxi\|_{2+\alpha}$. This and \eqref{eqn.Eh} imply $T_8 \lesssim h_{\max}\|\bxi\|_{2+\alpha}|\Theta_h|_{2,h}$. The same arguments in the proof of Lemma~\ref{lem:Ahproperties}.c provides $T_9 \lesssim \|G\||\bxi_I^h|_{2,h}\lesssim \|G\|\|\bxi\|_{2+\alpha}$ with Lemma~\ref{lem.vI} in the last step. A substitution of the estimates $T_1,\cdots,T_9$ in \eqref{eqn.t1tot9} reads
$$|E_h\Theta_h|_{1}^2\lesssim (h_{\max}^\alpha|\Theta_h|_{2,h}+\|G\|)\|\bxi\|_{2+\alpha}.$$
Since $\|Q\|_{-1}=|E_h\Theta_h|_{1,2,h}$, the above displayed inequality and \eqref{eqn.boundduallinearised} imply
$$|E_h\Theta_h|_{1}\lesssim h_{\max}^\alpha|\Theta_h|_{2,h}+\|G\|.$$
The combination of the above displayed inequality, \eqref{eqn.tri.thetah}, and \eqref{eqn.thetah} results in
$$|\Theta_h|_{2,h}\le C_1 h_{\max}^\alpha |\Theta_h|_{2,h}+C\|G\|.$$
For a choice of $h_{\max}\le h_1=\left(\frac{1}{2C_1}\right)^{\frac{1}{\alpha}}$, $|\Theta_h|_{2,h}\lesssim \|G\|$ and this implies the assertion.
\end{proof}
\begin{rem}
	Let $\Psi$ be a regular solution to \eqref{eqn.discretelinearised}. Then for a sufficiently small $h_{\max}$, the discrete linearised dual problem: given $Q \in \bL^2(\O)$, find $\bxi_h \in \bv_h$ such that
	$$\cA_{h}(\Phi_h,\bxi_h)=Q_h(\Phi_h), \forall\, \Phi_h \in \bV_h$$
	is well-posed, where $\cA_{h}(\bullet,\bullet)$ is defined as in \eqref{eqn.discretelinearised}. The proof is similar to Theorem~\ref{thm.discretelinearised} and hence is skipped.
\end{rem}
\noindent Since the discrete linearised problem and the dual problem are well-posed, $\cA_{h}:\bv_h\times \bv_h \to \R$ defined by
\begin{equation}\label{defn.Apsih}
\cA_{h}(\Theta_h,\Phi_h):=A_h(\Theta_h,\Phi_h)+\hB_h(\Psi,\Theta_h,\Phi_h)+\hB_h(\Theta_h,\Psi,\Phi_h)
\end{equation}
satisifies {\it discrete inf-sup condition} on $\bv_h \times \bv_h$ \cite{Brezzi,ng2}, that is, there exists a constant $\widehat{\beta}>0$ such that
\begin{equation}\label{defn.beta}
\sup_{|\Theta_h|_{2,h}=1}\cA_{h}(\Theta_h,\Phi_h)\ge \widehat{\beta} |\Phi_h|_{2,h}\mbox{ and } \sup_{|\Phi_h|_{2,h}=1}\cA_{h}(\Theta_h,\Phi_h)\ge \widehat{\beta} |\Theta_h|_{2,h}.
\end{equation}
\noindent Define the perturbed bilinear form by, for all $\Theta_h,\Phi_h \in \bv_h$,
\begin{equation}\label{defn.perturbed}
\hcA_{h}(\Theta_h,\Phi_h):=A_h(\Theta_h,\Phi_h)+B_h(\Psi_I^h,\Theta_h,\Phi_h)+B_h(\Theta_h,\Psi_I^h,\Phi_h),
\end{equation}
where $\Psi_I^h$ is the interpolant of $\Psi$.
\begin{lem}\label{lem:perturbed}
Let $\Psi$ be a regular solution to \eqref{VKE_weak}. Then for a sufficiently small $h_{\max}$, the perturbed bilinear form $\hcA_h(\bullet,\bullet)$ in \eqref{defn.perturbed} satisfies discrete inf-sup condition on $\bv_h \times \bv_h$.	
\end{lem}
\begin{proof}
Elementary algebra and \eqref{defn.beta} show
\begin{align}
\sup_{|\Theta_h|_{2,h}=1}\hcA_{h}&(\Theta_h,\Phi_h)=\sup_{|\Theta_h|_{2,h}=1}(A_h(\Theta_h,\Phi_h)+B_h(\Psi_I^h,\Theta_h,\Phi_h)+B_h(\Theta_h,\Psi_I^h,\Phi_h))\\
&\ge \sup_{|\Theta_h|_{2,h}=1}\cA_h(\Theta_h,\Phi_h)-\sup_{|\Theta_h|_{2,h}=1}(\hB_h({\Psi}-\Psi_I^h,\Theta_h,\Phi_h)+\hB_h(\Theta_h,{\Psi}-\Psi_I^h,\Phi_h))\\
&\ge \widehat{\beta} |\Phi_h|_{2,h}-\sup_{|\Theta_h|_{2,h}=1}(\hB_h({\Psi}-\Psi_I^h,\Theta_h,\Phi_h)+\hB_h(\Theta_h,{\Psi}-\Psi_I^h,\Phi_h)).\label{eqn.sup}
\end{align}
Lemma~\ref{lem:boundednessBh}.b,~\ref{lem:discreteSobolev}, and \eqref{eqn.Pihbound} imply $\hB_h({\Psi}-\Psi_I^h,\Theta_h,\Phi_h) \lesssim |{\Psi}-\Pi^h\Psi_I^h|_{2,h}|\Phi_h|_{2,h}$. 
This and \eqref{eqn.xiPixiIh1h} provide $\hB_h({\Psi}-\Psi_I^h,\Theta_h,\Phi_h) \le C_b^1h_{\max}^\alpha\|\Psi\|_{2+\alpha}|\Phi_h|_{2,h}$. Lemma~\ref{lem:boundednessBh}.b,~\ref{lem:discreteSobolev}, and \eqref{eqn.Pihbound} imply $\hB_h(\Theta_h,{\Psi}-\Psi_I^h,\Phi_h) \lesssim |{\Psi}-\Pi^h\Psi_I^h|_{1,4,h}|\Phi_h|_{2,h}$. The relation $\Pi^h\Psi_\pi^h=\Psi_\pi^h$, inverse estimate for $\Pi^h({\Psi_\pi^h}-\Psi_I^h )\in \P_2(\cT_h)$\cite{Brenner}, and Lemma~\ref{lem.vpi} show
\begin{align}\label{eqn.PsiPiPsiIh14h}
|{\Psi}-\Pi^h\Psi_I^h|_{1,4,h}&\le |{\Psi}-\Psi_\pi^h|_{1,4,h}+|\Pi^h{(\Psi_\pi^h}-\Psi_I^h)|_{1,4,h} \nonumber
\\&\lesssim  h_{\max}^{\alpha+1/2}\|\Psi\|_{2+\alpha}+h_{\max}^{-1/2}|\Pi^h({\Psi_\pi^h}-\Psi_I^h)|_{1,h}\nonumber\\
&\lesssim  h_{\max}^{\alpha+1/2}\|\Psi\|_{2+\alpha}+h_{\max}^{-1/2}(|{\Psi_\pi^h}-\Psi_I^h|_{1,h}+|\Psi_I^h-\Pi^h\Psi_I^h|_{1,h})\nonumber\\
&\lesssim h_{\max}^{\alpha+1/2}\|\Psi\|_{2+\alpha}+h_{\max}^{-1/2}|\Psi_I^h-\Pi^h\Psi_I^h|_{1,h}
\end{align}
with Lemma~\ref{lem.vpi} and~\ref{lem.vI} in the last step. 
 This and \eqref{t3} reveal 
 \begin{equation}\label{eqn.psipihppsiIh14h}
 |{\Psi}-\Pi^h\Psi_I^h|_{1,4,h} \le h_{\max}^{\alpha+1/2}\|\Psi\|_{2+\alpha}.
 \end{equation}
 \noindent Consequently, $\hB_h(\Theta_h,{\Psi}-\Psi_I^h,\Phi_h) \le C_b^2h_{\max}^\alpha\|\Psi\|_{2+\alpha}|\Phi_h|_{2,h}$. Let $C_b:=\max\{C_b^1,C_b^2\}$. Then, \eqref{eqn.sup} shows
\begin{align}
\sup_{|\Theta_h|_{2,h}=1}\hcA_{h}&(\Theta_h,\Phi_h)\ge \widehat{\beta} |\Phi_h|_{2,h}-2C_bh_{\max}^\alpha\|\Psi\|_{2+\alpha}|\Phi_h|_{2,h}.
\end{align}
A choice of $\displaystyle h_{\max}\le h_2=\left(\frac{\widehat{\beta}}{4C_b\|\Psi\|_{2+\alpha}}\right)^{1/\alpha}$ provides
\[\sup_{|\Theta_h|_{2,h}=1}\hcA_{h}(\Theta_h,\Phi_h)\ge ({\widehat{\beta}}/{2}) |\Phi_h|_{2,h}.\]
The analog arguments leads to $\displaystyle \sup_{|\Phi_h|_{2,h}=1}\hcA_{h}(\Phi_h,\Theta_h)\ge ({\widehat{\beta}}/{2}) |\Theta_h|_{2,h}.$
\end{proof}
\subsection{Existence and uniqueness}\label{sec:existence}
\noindent This section establishes the existence and uniqueness of the discrete solution and is an application of Brouwer’s fixed point theorem.
\begin{thm}[Existence and uniqueness of a discrete solution]\label{thm:existence}
	Let $\Psi$ be a regular solution to \eqref{VKE_weak}. For a sufficiently small choice of $h_{\max}$, there exists a unique solution $\Psi_h$ to the discrete \eqref{VKE_Morleyweak} that satisfies $|\Psi_h-\Psi_I^h|_{2,h}\le R(h_{\max})$ for some positive constant $R(h_{\max})$ depending on $h_{\max}$.
\end{thm}
\begin{proof}
\noindent Consider the non-linear mapping $T_h:\bV_h \to\bV_h$ defined by, for all $\Phi_h \in \bV_h$,
\begin{align}\label{eqn.Th}
&\hcA_h(T_h(\Theta_h),\Phi_h)=F_h(\Phi_h)+B_h(\Psi_I^h,\Theta_h,\Phi_h)+B_h(\Theta_h,\Psi_I^h,\Phi_h)-B_h(\Theta_h,\Theta_h,\Phi_h).
\end{align}
\noindent Lemma~\ref{lem:perturbed} shows that the operator $T_h$ is well defined and continuous. It is easy to check that the any discrete solution $\Psi_h$ to \eqref{VKE_Morleyweak} is a fixed point of $T_h$ and vice versa. Hence, in order to show the existence of a solution to \eqref{VKE_Morleyweak}, it is enough to prove that $T_h$ has a fixed point. For that, define
\[B_R(\Psi_I^h):=\{\Phi_h \in \bV_h:|\Phi_h-\Psi_I^h|_{2,h}\le R\}.\]
\noindent \emph{Step 1 establishes mapping of ball to ball. }
 Lemma~\ref{lem:perturbed} provides, for any $\Phi_h \in \bv_h$ with $|\Phi_h|_{2,h}=1,$
\begin{equation}\label{eqn.beta4}
{\widehat{\beta}}|T_h(\Theta_h)-\Psi_I^h|_{2,h} \le \hcA_h(T_h(\Theta_h)-\Psi_I^h,\Phi_h).
\end{equation}
The definition of $\hcA_h(\bullet,\bullet)$ (resp. $\hcA_h(T_h(\bullet),\bullet)$) in \eqref{defn.perturbed} (resp. \eqref{eqn.Th}) and \eqref{VKE_weak} with $\Phi=E_h\Phi_h$ show
\begin{align}
\hcA_h(T_h(\Theta_h)-\Psi_I^h,\Phi_h)&=\hcA_h(T_h(\Theta_h),\Phi_h)-\hcA_h(\Psi_I^h,\Phi_h)\nonumber\\
&=F_h(\Phi_h)+B_h(\Psi_I^h,\Theta_h,\Phi_h)+B_h(\Theta_h,\Psi_I^h,\Phi_h)-B_h(\Theta_h,\Theta_h,\Phi_h)\nonumber\\
&\quad -A_h(\Psi_I^h,\Phi_h)-2B_h(\Psi_I^h,\Psi_I^h,\Phi_h)\nonumber\\
&=(F_h(\Phi_h)-F(E_h\Phi_h))+(A(\Psi,E_h\Phi_h)-A_h(\Psi_I^h,\Phi_h))\nonumber\\
&\quad +(B(\Psi,\Psi,E_h\Phi_h)-B_h(\Psi_I^h,\Psi_I^h,\Phi_h))+B_h(\Psi_I^h-\Theta_h,\Theta_h-\Psi_I^h,\Phi_h)\nonumber\\
&=:T_1+T_2+T_3+T_4.\label{t1tot4}
\end{align}
An introduction of $(f,\varphi_{1,h})$ for $\Phi_h=(\varphi_{1,h},\varphi_{2,h})$,  $|\langle f_h,\varphi_{1,h}\rangle_K -(f,\varphi_{1,h})_K| \lesssim h_{K}\|f\|_{0,K}|\varphi_{1,h}|_{1,K}$ \cite[Lemma 5.2]{Zhao_Morley} for $K \in \cT_h$, Cauchy inequality, and Lemma~\ref{lem:Eh} lead to
\[T_1 = \langle f_h,\varphi_{1,h}\rangle -(f,\varphi_{1,h})+(f,\varphi_{1,h}-E_h\varphi_{1,h})\lesssim h_{\rm max}\|f\|.\]
The consistency property $A_h(\Psi_\pi^h,\Phi_h)=A_{\pw}(\Psi_\pi^h,\Phi_h)$ from \eqref{eqn.consistency} reveals $T_2 = A_{\pw}(\Psi,E_h\Phi_h-\Phi_h)-A_h(\Psi_I^h-\Psi_\pi^h,\Phi_h)+A_{\pw}(\Psi-\Psi_\pi^h,\Phi_h)$. Lemma~\ref{lem:Apw}.a and~\ref{lem:Ahproperties}.a, Cauchy inequality, \eqref{eqn.vIvpi}, and Lemma~\ref{lem.vpi} read
$T_2\lesssim h_{\rm max}^\alpha \|\Psi\|_{2+\alpha}.$
Lemma~\ref{lem:boundednessBh}.c, b, e,~\ref{lem:discreteSobolev}, and \eqref{eqn.Pihbound} show
\begin{align*}
T_3&=\hB_h(\Psi,\Psi,E_h\Phi_h-\Phi_h)+\hB_h(\Psi-\Psi_I^h,\Psi_I^h,\Phi_h)+\hB_h(\Psi,\Psi-\Psi_I^h,\Phi_h)\nonumber\\
&\lesssim \|\Psi\|_{2+\alpha}\|\Psi\|_2|E_h\Phi_h-\Pi^h\Phi_h|_{1,h}+|\Psi-\Pi^h\Psi_I^h|_{2,h}|\Psi_I^h|_{2,h}+\|\Psi\|_{2+\alpha}|\Psi-\Pi^h\Psi_I^h|_{1,h}\nonumber .
\end{align*}
\noindent The estimates in \eqref{eqn.xiPixiIh1h}, \eqref{eqn.Eh}, and Lemma~\ref{lem.vI} in the above displayed inequality shows $T_3 \lesssim h_{\max}^\alpha\|\Psi\|_{2+\alpha}^2$. Lemma~\ref{lem:boundednessBh}.a (with hidden constant $C_b$) implies
$T_4 \le C_b |\Psi_I^h-\Theta_h|_{2,h}^2.$
 A combination of $T_1,\cdots,T_4$ in \eqref{t1tot4} and then in \eqref{eqn.beta4} leads to
\[|T_h(\Theta_h)-\Psi_I^h|_{2,h} \le Ch_{\rm max}^\alpha+C_b|\Psi_I^h-\Theta_h|_{2,h}^2\]
for some positive constant $C$ independent of $h$ but dependent on $\|\Psi\|_{2+\alpha}$ and $\|f\|$. A choice of $h_{\max}\le h_3=\left(\frac{1}{4C C_b}\right)^{\frac{1}{\alpha}}$ yields $4CC_bh_{\max}^\alpha \le 1$. Since $|\Theta_h-\Psi_I^h|_{2,h}\le R(h_{\max})$, a choice of $R(h_{\max})=2Ch_{\max}^\alpha$ leads to
\[|\Theta_h-\Psi_I^h|_{2,h}\le Ch_{\rm max}^\alpha(1+4CC_bh_{\rm max}^\alpha)\le R(h_{\max}).\]
Hence, $T_h$ maps the ball $B_R(\Psi_I^h)$ into itself.

\smallskip

\noindent \emph{Step 2 establishes the existence of a discrete solution. }Since $T_h$ maps $B_R(\Psi_I^h)$ to itself from Step 1, the Brouwer fixed
	point theorem yields that the mapping $T_h$ has a fixed point, say $\Psi_h$. Hence, \eqref{eqn.Th} and \eqref{defn.perturbed} reveal that $\Psi_h$ is a solution to \eqref{VKE_Morleyweak} with $|\Psi_h-\Psi_I^h|_{2,h}\le R(h_{\max})$.

\smallskip

\noindent \emph{Step 3 establishes that $T_h$ is a contraction. }For $\Theta_1,\Theta_2 \in B_{R(h_{\max})}(\Psi_I^h)$ and for all $\Phi_h \in \bV_h$, let $T_h(\Theta_i)$, $i=1,2$ be the solutions of
		\begin{align}
		\hcA_h(T_h(\Theta_i),\Phi_h)&=F_h(\Phi_h)+B_h(\Psi_I^h,\Theta_i,\Phi_h)+B_h(\Theta_i,\Psi_I^h,\Phi_h)-B_h(\Theta_i,\Theta_i,\Phi_h).
		\end{align}
\noindent Lemma~\ref{lem:perturbed} shows, for any $\Phi_h \in \bv_h$ with $|\Phi_h|_{2,h}=1,$
\begin{align}
{\widehat{\beta}}&|T_h(\Theta_1)-T_h(\Theta_2)|_{2,h} \le \hcA_h(T_h(\Theta_1)-T_h(\Theta_2),\Phi_h)\\
&=B_h(\Psi_I^h,\Theta_1-\Theta_2,\Phi_h)+B_h(\Theta_1-\Theta_2,\Psi_I^h,\Phi_h)+B_h(\Theta_2,\Theta_2,\Phi_h)-B_h(\Theta_1,\Theta_1,\Phi_h)\\
&=B_h(\Theta_2-\Theta_1,\Theta_1-\Psi_I^h,\Phi_h)+B_h(\Theta_2-\Psi_I^h,\Theta_2-\Theta_1,\Phi_h)\\
&\lesssim |\Theta_2-\Theta_1|_{2,h}(|\Theta_1-\Psi_I^h|_{2,h}+|\Theta_2-\Psi_I^h|_{2,h})
\end{align}		
with Lemma~\ref{lem:boundednessBh}.a in the last step. Since $\Theta_1,\Theta_2 \in B_{R(h_{\max})}(\Psi_I^h)$, the above displayed estimate provides
\[|T_h(\Theta_1)-T_h(\Theta_2)|_{2,h} \le Ch_{\max}^\alpha|\Theta_1-\Theta_2|_{2,h}\]
for some positive constant $C$ independent of $h_{\max}$.

\smallskip

\noindent \emph{Step 4 establishes local uniqueness of a discrete solution. }Let $\Psi$ be a regular solution to \eqref{VKE_weak}. For a sufficiently small choice of $h_{\max}$, the contraction mapping theorem establishes the local uniquenes  of the discrete solution $\Psi_h$ to \eqref{VKE_Morleyweak}.
\end{proof}
\begin{lem}[Bound for $\Psi_h$]\label{lem:Psih}
	Let $\Psi_h$ be a discrete solution to \eqref{VKE_Morleyweak} that satisfies $|\Psi_h-\Psi_I^h|_{2,h}\le R(h_{\max})$. Then $|\Psi_h|_{2,h} \lesssim \|f\|$.
\end{lem}
\begin{proof}
\noindent Triangle inequalities, $|\Psi_h-\Psi_I^h|_{2,h}\le R(h_{\max})$, and Lemma~\ref{lem.vI} provide
	\begin{align}
	|\Psi_h|_{2,h}&\le |\Psi_h-\Psi_I^h|_{2,h}+|\Psi-\Psi_I^h|_{2,h}+|\Psi|_2\lesssim R(h_{\max})+ h_{\max}^\alpha\|\Psi\|_{2+\alpha}+|\Psi|_2.
	\end{align}
Since $R(h)=2Ch_{\max}^\alpha$ with $C$ depends on $\|\Psi\|_{2+\alpha}$ and $\|f\|$ from Step1 of Theorem~\ref{thm:existence}, a combination of Lemma~\ref{lem:aprioriboundcts} and the above displayed estimate concludes the proof.	
\end{proof}

\section{A priori error control}\label{sec:error}
\noindent  This section deals with the a priori error control under minimal regularity assumption on the exact solution and is followed by the convergence of the Newtons method.
\subsection{Error estimates}
This section proves an a priori error estimate in $H^2$ and $H^1$ norms.
\begin{thm}[Energy norm estimate]\label{thm:energyerror}
		Let $\Psi$ be a regular solution to \eqref{VKE_weak}. For a sufficiently small $h_{\max}$, the discrete solution $\Psi_h$ to \eqref{VKE_Morleyweak} satisfies
		$$\|\Psi-\Psi_h\|_{2,h} \lesssim h_{\max}^\alpha,$$
		where $\alpha \in (\half,1]$ is the index of elliptic regularity.
\end{thm}
\begin{proof}
	The triangle inequality, Lemma~\ref{lem.vI}, \eqref{eqn.seminormbound}, and Theorem~\ref{thm:existence} for sufficiently small $h_{\max}$ with $R(h_{\max}) \lesssim h_{\max}^\alpha$ show
	\begin{align}
	\|\Psi-\Psi_h\|_{2,h}&\le \|\Psi-\Psi_I^h\|_{2,h}+\|\Psi_I^h-\Psi_h\|_{2,h} \lesssim h_{\max}^\alpha.
	\end{align}
\end{proof}
\noindent To prove the lower order error estimates, define the augmented local space \cite{Zhao_Morley}
\begin{align}\label{def.whk}
W_h(K):=\displaystyle\bigg\{&\varphi \in H^2(K);\Delta^2\varphi \in \P_0(K), \varphi_{|e} \in \P_2(e), \Delta \varphi_{|e} \in \P_{0}(e), \int_e \Pi^K\varphi\ds=\int_e \varphi\ds \\
&  \; \forall e \subseteq \partial K,\int_K \Pi^K\varphi\dx=\int_K \varphi\dx\bigg\}.
\end{align}
\noindent The global space $W_h$ can be then assembled in the same way as $V_h$ is defined. 
As in \cite[Section 6]{Zhao_Morley}, reformulate the nonconforming VEM \eqref{VKE_Morleyweak} as
\begin{equation}\label{VKE_Morleyweak_modified}
A_h(\Psi_h,\Phi_h)+B_h(\Psi_h,\Psi_h,\Phi_h)=(f_h,\varphi_{1,h})=: L_h(\Phi_h) \quad \forall\,\Phi_h\in \bW_h
\end{equation}
with the virtual element space $W_h$ instead of $V_h$ and $f_h|_K:=P_0^K f$. The same arguments in Theorem~\ref{thm:energyerror} leads to
\begin{equation}\label{eqn.energyerror.modified}
\|\Psi-\Psi_h\|_{2,h} \lesssim h_{\max}^\alpha.
\end{equation}
Also, for $f \in H^1(\O)$ and $\varphi_h \in W_h$, the following approximation property hold \cite[(6.2)]{Zhao_Morley}:
\begin{equation}\label{eqn.ffh}
(f-f_h,\varphi_{I}^{h})\lesssim h_{\max}^2\|f\|_1\|\varphi_I^h\|_{1,h}.
\end{equation}
\begin{thm}[$H^1$ estimate]\label{thm:lowererror}
	Let $\Psi$ be a regular solution to \eqref{VKE_weak} and $f \in H^1(\O)$. For a sufficiently small $h_{\max}$, the discrete solution $\Psi_h$ to \eqref{VKE_Morleyweak_modified} satisfies
	$$\|\Psi-\Psi_h\|_{1,h} \lesssim h_{\max}^{2\alpha},$$
	where $\alpha \in (\half,1]$ is the index of elliptic regularity.
\end{thm}
\begin{proof}
The triangle inequality and Lemma~\ref{lem.vI} lead to
\begin{align}\label{eqn.H1tri}
	\|\Psi-\Psi_h\|_{1,h}&\le \|\Psi-\Psi_I^h\|_{1,h}+\|\Psi_I^h-\Psi_h\|_{1,h}\nonumber\\
&\lesssim h_{\max}^{1+\alpha}+\|\rho_h-E_h\rho_h\|_{1,h}+\|E_h\rho_h\|_{1}
\end{align}	
with $\rho_h:=\Psi_I^h-\Psi_h \in \bv_h$ in the last step. The triangle inequality, Lemma~\ref{lem.vI}, and \eqref{eqn.energyerror.modified} show
 \begin{equation}\label{eqn.rhoh}
\|\rho_h\|_{2,h}\le \|\Psi_I^h-\Psi\|_{2,h}+\|\Psi-\Psi_h\|_{2,h}\lesssim h_{\max}^\alpha.
\end{equation}
A combination of \eqref{eqn.H1tri}, Lemma~\ref{lem:Eh}, and \eqref{eqn.rhoh} reads
\begin{align}\label{eqn.H1tri1}
\|\Psi-\Psi_h\|_{1,h}&\lesssim h_{\max}^{1+\alpha}+\|E_h\rho_h\|_{1}.
\end{align}	
The choice $Q=-\Delta E_h\rho_h$ and $\Phi=E_h\rho_h$ in \eqref{eqn.dualcts} and elementary algebra reveal
\begin{align}\label{eqn.H1nabla}
\|\nabla E_h\rho_h\|_{}^2 &=\cA(E_h\rho_h,\bxi)=A(E_h\rho_h,\bxi)+B(\Psi,E_h\rho_h,\bxi)+B(E_h\rho_h,\Psi,\bxi)\nonumber\\
&=A_{\pw}(E_h\rho_h-\rho_h,\bxi)+A_{\pw}(\rho_h,\bxi)+\hB_h(\Psi,E_h\rho_h-\rho_h,\bxi)\nonumber\\
&\quad +\hB_h(E_h\rho_h-\rho_h,\Psi,\bxi)+\hB_h(\Psi,\rho_h,\bxi)+\hB_h(\rho_h,\Psi,\bxi)\nonumber\\
&=A_{\pw}(E_h\rho_h-\rho_h,\bxi)+\hB_h(\Psi,E_h\rho_h-\rho_h,\bxi) +\hB_h(E_h\rho_h-\rho_h,\Psi,\bxi)\nonumber\\
&\quad + A_{\pw}(\Psi_I^h-\Psi,\bxi)+A_{\pw}(\Psi-\Psi_h,\bxi-\bxi_I^h)+A_{\pw}(\Psi,\bxi_I^h-\bxi)\nonumber\\
&\quad +A(\Psi,\bxi)-A_{\pw}(\Psi_h,\bxi_I^h)+\hB_h(\Psi,\rho_h,\bxi) +\hB_h(\rho_h,\Psi,\bxi).
\end{align}
Consider the third last term on the right-hand side of \eqref{eqn.H1nabla}. Since $\Psi_\pi^h, \bxi_\pi^h \in \P_2(\cT_h)$, the consistency property \eqref{eqn.consistency} reads $A_{\pw}(\Psi_\pi^h,\bxi_I^h-\bxi_\pi^h)=A_{h}(\Psi_\pi^h,\bxi_I^h-\bxi_\pi^h)$ and $A_{\pw}(\Psi_h,\bxi_\pi^h)=A_{h}(\Psi_h,\bxi_\pi^h)$. Hence,
 \begin{align}
A_{\pw}(\Psi_h,\bxi_I^h)&=A_{\pw}(\Psi_h-\Psi_\pi^h,\bxi_I^h-\bxi_\pi^h)+A_{\pw}(\Psi_\pi^h,\bxi_I^h-\bxi_\pi^h)+A_{\pw}(\Psi_h,\bxi_\pi^h)\\
&=A_{\pw}(\Psi_h-\Psi_\pi^h,\bxi_I^h-\bxi_\pi^h)+A_{h}(\Psi_\pi^h-\Psi_h,\bxi_I^h-\bxi_\pi^h)+A_{h}(\Psi_h,\bxi_I^h)\\
&=A_{\pw}(\Psi_h-\Psi_\pi^h,\bxi_I^h-\bxi_\pi^h)+A_{h}(\Psi_\pi^h-\Psi_h,\bxi_I^h-\bxi_\pi^h)+L_{h}(\bxi_I^h)-B_h(\Psi_h,\Psi_h,\bxi_I^h)
\end{align}
with \eqref{VKE_Morleyweak_modified} for $\Phi_h=\bxi_I^h$ in the last step. This, \eqref{eqn.H1nabla}, and \eqref{VKE_weak} for $\Phi=\bxi$ result in
\begin{align}\label{eqn.H1nabla.1}
\|\nabla E_h\rho_h\|_{}^2 &=A_{\pw}(E_h\rho_h-\rho_h,\bxi)+\hB_h(\Psi,E_h\rho_h-\rho_h,\bxi) +\hB_h(E_h\rho_h-\rho_h,\Psi,\bxi)\nonumber\\
&\quad + A_{\pw}(\Psi_I^h-\Psi,\bxi)+A_{\pw}(\Psi-\Psi_h,\bxi-\bxi_I^h)+A_{\pw}(\Psi,\bxi_I^h-\bxi) +(F(\bxi)-L_{h}(\bxi_I^h))\nonumber\\
&\quad-A_{\pw}(\Psi_h-\Psi_\pi^h,\bxi_I^h-\bxi_\pi^h)-A_{h}(\Psi_\pi^h-\Psi_h,\bxi_I^h-\bxi_\pi^h) +(B_h(\Psi_h,\Psi_h,\bxi_I^h)\nonumber\\
&\quad-B(\Psi,\Psi,\bxi)+\hB_h(\Psi,\Psi_I^h-\Psi_h,\bxi) +\hB_h(\Psi_I^h-\Psi_h,\Psi,\bxi)):=\sum_{i=1}^{10}T_i.
\end{align}
\noindent Lemma~\ref{lem:Apw}.a and \eqref{eqn.rhoh} provide $T_1 \lesssim h_{\max}^\alpha |\rho_h|_{2,h}\|\bxi\|_{2+\alpha}\lesssim h_{\max}^{2\alpha}\|\bxi\|_{2+\alpha}$. Lemma~\ref{lem:boundednessBh}.e and \eqref{eqn.Eh} imply
\begin{align}
T_2 &\lesssim \|\Psi\|_{2+\alpha}|E_h\rho_h-\Pi^h\rho_h|_{1,h}|\bxi|_2
\lesssim h_{\max}|\rho_h|_{2,h}|\bxi|_2 \lesssim h_{\max}^{1+\alpha}|\bxi|_2
\end{align}
with \eqref{eqn.rhoh} in the last step. Since $\Psi,\bxi \in \bV \cap \bH^{2+\alpha}(\O)$, the Sobolev embedding $H^{2+\alpha}(\O) \hookrightarrow W^{2,4}(\O)$ shows
\[|\nabla \Psi \cdot \nabla \bxi|_{1} \le |\nabla \Psi|_{1,4}|\nabla \bxi|_{1,4}\lesssim \|\Psi\|_{2+\alpha}\|\bxi\|_{2+\alpha}.\]
Hence, $\nabla \Psi \cdot \nabla \bxi \in H^1(\O)$ and $T_3 \lesssim |D^2(E_h\rho_h-\Pi^h\rho_h)|_{-1,h}|\nabla \Psi \cdot \nabla \bxi|_{1,h}$. The definition of the dual norm and integration by parts lead to $|D^2(E_h\rho_h-\Pi^h\rho_h)|_{-1,h}\lesssim |E_h\rho_h-\Pi^h\rho_h|_{1,h}$. The combination of the above estimates, \eqref{eqn.Eh}, and \eqref{eqn.rhoh} implies $T_3 \lesssim h_{\max}^{1+\alpha}\|\bxi\|_{2+\alpha}$. Lemma~\ref{lem:Apw}.b reads $T_4+T_6 \lesssim h_{\rm max}^{2\alpha}\|\bxi\|_{2+\alpha}$. The boundedness of $A_{\pw}(\bullet,\bullet)$, \eqref{eqn.energyerror.modified}, and Lemma~\ref{lem.vI} provide $T_5 \lesssim  h_{\rm max}^{2\alpha}\|\bxi\|_{2+\alpha}$. The Cauchy inequality, Lemma~\ref{lem.vI}, and \eqref{eqn.ffh} show
\begin{align}
T_7&=(f,\bxi_1-\bxi_{1,I}^{h})+(f-f_h,\bxi_{1,I}^{h})\lesssim h_{\max}^{2+\alpha}\|\bxi\|_{2+\alpha}+h_{\max}^2\|\bxi_I^h\|_{1,h} \lesssim h_{\max}^{2}\|\bxi\|_{2+\alpha}.
\end{align}
\noindent Triangle inequality with $\Psi$, \eqref{eqn.energyerror.modified}, and Lemma~\ref{lem.vpi} provide $|\Psi_h-\Psi_\pi^h|_{2,h} \lesssim h_{\max}^{\alpha}$. This, the boundedness of $A_{\pw}(\bullet,\bullet)$, Lemma~\ref{lem:Ahproperties}.a, and \eqref{eqn.vIvpi} imply $T_8+T_9 \lesssim h_{\max}^{2\alpha}\|\bxi\|_{2+\alpha}$. Elementary algebra results in
\begin{align}
T_{10}
&=B_h(\Psi_h,\Psi_h,\bxi_I^h)-B(\Psi,\Psi,\bxi)+\hB_h(\Psi,\Psi_I^h-\Psi,\bxi) +\hB_h(\Psi_I^h-\Psi,\Psi,\bxi)\nonumber\\
&\qquad +\hB_h(\Psi,\Psi-\Psi_h,\bxi) +\hB_h(\Psi-\Psi_h,\Psi,\bxi)\\
&=\hB_h(\Psi,\Psi_I^h-\Psi,\bxi) +\hB_h(\Psi_I^h-\Psi,\Psi,\bxi) +\hB_h(\Psi-\Psi_h,\Psi-\Psi_h,\bxi)\nonumber\\
&\quad +\hB_h(\Psi_h-\Psi,\Psi_h,\bxi_I^h-\bxi)+\hB_h(\Psi,\Psi_h,\bxi_I^h-\bxi).\label{t10}
\end{align}
Lemma~\ref{lem:boundednessBh}.e and \eqref{eqn.xiPixiIh1h}
  reveal
$\hB_h(\Psi,\Psi_I^h-\Psi,\bxi) \lesssim h_{\max}^{1+\alpha}\|\bxi\|_{2}$. Arguments analogous to $T_3$ together with \eqref{eqn.xiPixiIh1h} lead to $\hB_h(\Psi_I^h-\Psi,\Psi,\bxi)\lesssim h_{\max}^{1+\alpha}\|\bxi\|_{2+\alpha}$. The relation $\Psi_\pi^h=\Pi^h\Psi_\pi^h$, Lemma~\ref{lem.vpi}, and Lemma~\ref{lem:discreteSobolev} imply 
\begin{equation}\label{eqn.psipihpsih2h}
|\Psi-\Pi^h\Psi_h |_{1,h} \le |\Psi-\Psi_\pi^h |_{1,h}+|\Pi^h(\Psi_\pi^h-\Psi_h) |_{1,h}\lesssim h_{\max}^{1+\alpha}+|\Pi^h(\Psi_\pi^h-\Psi_h) |_{2,h}\lesssim h_{\max}^\alpha
\end{equation} 
with \eqref{eqn.Pihbound}, Lemma~\ref{lem.vpi}, and \eqref{eqn.energyerror.modified} in the last step. The same arguments provides $|\Psi-\Pi^h\Psi_h |_{2,h} \lesssim h_{\max}^\alpha$. This, \eqref{eqn.psipihpsih2h}, the symmetry of the $\hB_h(\bullet,\bullet,\bullet)$ with respect to the second and third arguments, and Lemma~\ref{lem:boundednessBh}.f show $\hB_h(\Psi-\Psi_h,\Psi-\Psi_h,\bxi) \lesssim h_{\max}^{2\alpha}\|\bxi\|_{2+\alpha}$. 
\noindent Lemma~\ref{lem:boundednessBh}.b, the estimate $|\Psi-\Pi^h\Psi_h |_{2,h} \lesssim h_{\max}^\alpha$ from the previous term, Lemma~\ref{lem:discreteSobolev},\eqref{eqn.Pihbound}, Lemma~\ref{lem:Psih}, and the same arguments in \eqref{eqn.psipihppsiIh14h} for $\bxi$ lead to $\hB_h(\Psi_h-\Psi,\Psi_h,\bxi_I^h-\bxi)\lesssim h_{\max}^{2\alpha+1/2}\|\bxi\|_{2+\alpha}.$ Lemma~\ref{lem:boundednessBh}.d,~\ref{lem:Psih}, and \eqref{eqn.xiPixiIh1h} read  $\hB_h(\Psi,\Psi_h,\bxi_I^h-\bxi) \lesssim h_{\max}^{1+\alpha}\|\bxi\|_{2+\alpha}$. A combination of these estimates in \eqref{t10} leads to $T_{10} \lesssim h_{\max}^{2\alpha}\|\bxi\|_{2+\alpha}.$ The substitution of $T_1,\cdots,T_{10}$ in \eqref{eqn.H1nabla.1} provides
$$\|\nabla E_h\rho_h\|_{}^2\lesssim h_{\max}^{2\alpha}\|\bxi\|_{2+\alpha}.$$
Since $\|\bxi\|_{2+\alpha} \lesssim \|\Delta E_h\rho_h\|_{-1} \lesssim \|\nabla E_h\rho_h\|$ from \eqref{eqn.boundduallinearised} and integration by parts, $\|\nabla E_h\rho_h\|_{}\lesssim h_{\max}^{2\alpha}$. This with \eqref{eqn.H1tri1} concludes the proof.
\end{proof}
\begin{thm}[$L^2$ estimate]\label{thm:lowererrorl2}
	Let $\Psi$ be a regular solution to \eqref{VKE_weak} and $f \in H^1(\O)$. For sufficiently small $h_{\max}$, the discrete solution $\Psi_h$ to \eqref{VKE_Morleyweak_modified} satisfies
	$$\|\Psi-\Psi_h\| \lesssim h_{\max}^{2\alpha},$$
	where $\alpha \in (\half,1]$ is the index of elliptic regularity.
\end{thm}
\begin{proof}
	\noindent The triangle inequalities lead to
\begin{align}
\|\Psi-\Psi_h\|&\le \|\Psi-\Psi_I^h\|+\|\Psi_I^h-\Psi_h\|\le  \|\Psi-\Psi_I^h\|+\|\rho_h-E_h\rho_h\|+\|E_h\rho_h\|
\end{align}
with $\rho_h= \Psi_I^h-\Psi_h$ in the last step. Lemma~\ref{lem.vI} shows $\|\Psi-\Psi_I^h\|\lesssim h_{\max}^{2+\alpha}$. Lemma~\ref{lem:Eh} and \eqref{eqn.rhoh} provide $\|\rho_h-E_h\rho_h\| \lesssim h_{\max}^2|\rho_h|_{2,h} \lesssim h_{\max}^{2+\alpha}$. Since $E_h\rho_h \in H^1_0(\O)$, $\|E_h\rho_h\| \lesssim \|\nabla E_h\rho_h\|$. This and Theorem~\ref{thm:lowererror} reveal $\|E_h\rho_h\|\lesssim h_{\max}^{2\alpha}$. A combination of these estimates concludes the proof.
\end{proof}
\begin{rem}[$L^2$ error estimate] It is well-known \cite{HuShi,ng2} that for the Morley nonconforming finite element method for the biharmonic problem and \vket, the $L^2$ error estimate cannot be improved compared to that of $H^1$  error estimate. Since Morley finite element method is a special case of simplified fully nonconforming VEM (see Remark~\ref{rem.morley}), it is expected that using a lower order VEM, the order of convergence in $L^2$ norm cannot be improved than that of the $H^1$ norm. Also, the same result for this VEM for the biharmonic problem  is presented in \cite[Theorem 6.2]{Zhao_Morley}.
\end{rem}
\subsection{Convergence of the Newtons method}\label{sec:Newton}
The discrete solution $\Psi_{h}$ of \eqref{VKE_Morleyweak} is characterized as the fixed point of \eqref{eqn.Th} and so depends on the unknown $\Psi_I^h$. The approximate solution to \eqref{VKE_Morleyweak} is computed with the Newton method, where the iterates $\Psi_{h}^{j}$ solve 
\begin{equation}\label{NewtonIterate}
A_{h}(\Psi_{h}^{j},\Phi_{h})+B_{h}(\Psi_{h}^{j-1},\Psi_{h}^{j},\Phi_{h})+B_{h}(\Psi_{h}^{j},\Psi_{h}^{j-1},\Phi_{h})=B_{h}(\Psi_{h}^{j-1},\Psi_{h}^{j-1},\Phi_{h})+F_{h}(\Phi_{h})
\end{equation} 
for all $\Phi_{h}\in \bv_h$. The Newton method has locally quadratic convergence.
\begin{thm}[Convergence of Newton method]\label{NewtonThm}
	Let $\Psi$ be a regular solution to \eqref{VKE_weak} and let  $\Psi_{h}$ solve \eqref{VKE_Morleyweak}. There exists a positive constant $R$ independent of $h$, such that for any initial guess $\Psi_{h}^0$ with
	$\displaystyle | \Psi_{h}- \Psi_{h}^0|_{2,h}\leq  R$, it follows $|\Psi_{h}- \Psi_{h}^{j}|_{2,h}  \leq R\fl j=0,1,2,\ldots$ and the iterates of the Newton method  in \eqref{NewtonIterate} are well defined and converges quadratically to $\Psi_{h}$.
\end{thm}
\begin{proof}
The proof follows the lines of \cite[Theorem 6.2]{ng2}. However, for the sake of completeness, we provide a detailed proof.

\noindent  Lemma~\ref{lem:perturbed} shows that there exists a positive constant $\epsilon$ (sufficiently small) independent of $h$ such that for each $\bz_h \in \bv_h$ with $| \bz_h-\Psi_I^h|_{2,h}\leq \epsilon$, the bilinear form
	\begin{equation}\label{NewtonNonsingular}
	A_{h}(\bullet,\bullet)+B_{h}(\bz_h,\bullet,\bullet)+B_{h}(\bullet,\bz_h,\bullet)
	\end{equation}
	satisfies discrete inf-sup condition on $\bv_h \times \bv_h$.
	For sufficiently small $h_{\max}$, Theorem~\ref{thm:existence} with $R(h_{\max})\lesssim h_{\max}^\alpha$ implies $|\Psi_I^h-\Psi_h|_{2,h} \lesssim h_{\max}^\alpha$. Thus $h_{\max}$ can be chosen sufficiently small so that $|\Psi_I^h-\Psi_{h}|_{2,h}\leq \epsilon/2$. Recall ${\widehat{\beta}}$ from \eqref{defn.beta}. Let ${C_{b}}$ be the hidden positive constant in Lemma~\ref{lem:boundednessBh}.a. Set
	\begin{equation*}
	R:=\min\left\{\epsilon/2,{\widehat{\beta}}/8{C_{b}}\right\}.
	\end{equation*}
	Assume that the initial guess $\Psi_{h}^0$ satisfies $|\Psi_{h}-\Psi_{h}^0|_{2,h}\leq  R$. Then,
	\begin{equation*}
	|\Psi_I^h-\Psi_{h}^0|_{2,h}\leq |\Psi_I^h-\Psi_{h}|_{2,h}+|\Psi_{h}-\Psi_{h}^0|_{2,h}\leq \epsilon.
	\end{equation*}
	This implies $|\Psi_{h}-\Psi_{h}^{j-1}|_{2,h}\leq  R$ and $	|\Psi_I^h-\Psi_{h}^{j-1}|_{2,h} \leq \epsilon$ for $j=1$ and suppose for mathematical induction that this holds for some $j\in\bN$. Then $\bz_{h}:=\Psi_{h}^{j-1}$ in \eqref{NewtonNonsingular} leads to an discrete inf-sup condition of $A_{h}(\bullet,\bullet)+B_{h}(\Psi^{j-1}_{h},\bullet,\bullet)+B_{h}(\bullet,\Psi^{j-1}_{h},\bullet)$ and so to an unique solution $\Psi_{h}^j$ to \eqref{NewtonIterate} in step $j$ of the Newton scheme. The discrete inf-sup condition \eqref{NewtonNonsingular} implies the existence of $ \Phi_{h}\in\bv_h$ with $| \Phi_{h}|_{2,h}=1$ and
	\begin{equation*}
	\frac{{\widehat{\beta}}}{4}|\Psi_h-\Psi_h^{j}|_{2,h}\leq A_{h}(\Psi_h-\Psi_h^{j}, \Phi_{h})+B_{h}(\Psi_h^{j-1},\Psi_h-\Psi_{h}^{j}, \Phi_{h})+B_{h}(\Psi_h-\Psi_{h}^{j},\Psi_h^{j-1}, \Phi_{h}).
	\end{equation*}
	The application of \eqref{NewtonIterate}, \eqref{VKE_Morleyweak}, and Lemma~\ref{lem:boundednessBh}.a result in 
	\begin{align}
	&A_{h}(\Psi_h-\Psi_h^{j}, \Phi_{h})+B_{h}(\Psi_h^{j-1},\Psi_h-\Psi_{h}^{j}, \Phi_{h})+B_{h}(\Psi_h-\Psi_{h}^{j},\Psi_h^{j-1}, \Phi_{h})\notag\\
	&=A_{h}(\Psi_h,  \Phi_{h})+B_{h}(\Psi_h^{j-1},\Psi_h,  \Phi_{h})+B_{h}(\Psi_h, \Psi_h^{j-1}, \Phi_{h})-B_{h}(\Psi_h^{j-1},\Psi_h^{j-1},  \Phi_{h})-F_{h}(  \Phi_{h})\notag\\
	&=-B_{h}(\Psi_h,\Psi_h,\Phi_{\dg})+B_{h}(\Psi_h^{j-1},\Psi_h,  \Phi_{h})+B_{h}(\Psi_h, \Psi_h^{j-1}, \Phi_{h})-B_{h}(\Psi_h^{j-1},\Psi_h^{j-1},  \Phi_{h})\notag\\
	&=B_{h}(\Psi_h-\Psi_h^{j-1},\Psi_h^{j-1}-\Psi_h,  \Phi_{h})\leq {C_{b}}| \Psi_h-\Psi_h^{j-1}|_{2,h}^2.\notag
	\end{align}
	This implies 
	\begin{equation}\label{induction0}
	|\Psi_h-\Psi_h^{j}|_{2,h}\leq \left(4{C_{b}}/{\widehat{\beta}}\right)|\Psi_h-\Psi_h^{j-1}|_{2,h}^2
	\end{equation}
	and establishes the quadratic convergence of the Newton method  to $\Psi_h$. The definition of $R$, $|\Psi_{h}-\Psi_{h}^{j-1}|_{2,h}\leq  R$, and \eqref{induction0} guarantee $|\Psi_h-\Psi_h^{j}|_{2,h}\leq \half |\Psi_h-\Psi_h^{j-1}|_{2,h}<R$ to allow an induction step $j\to j+1$ to conclude the proof.
\end{proof}
\noindent Hence, the discrete non-linear problem can be solved using the Newton’s method by choosing an appropriate initial guess such that there exists a closed sphere in which the approximate solution is unique and the Newton’s iterates converge quadratically to the discrete solution.

\section{Numerical Results}\label{sec.numericalresults}

This section presents a few examples on general polygonal meshes to illustrate the theoretical estimates in the previous section. 

\smallskip

\noindent The solution $\Psi_h=(u_h,v_h)$ to \eqref{VKE_Morleyweak} is computed using Newtons method where the initial value for $\Psi_h$ in the iterative scheme is the discrete solution to the corresponding biharmonic problem without the trilinear term. The convergence of Newtons' method for the VEM scheme is proved in Theorem~\ref{NewtonThm}. The implementation associated with the trilinear term was done following the ideas in \cite[Section 5.1]{ng1} taking into account of the VEM approximation \cite{Veiga_hitchhikersVEM}. The numerical results are presented for square domain and L-shaped domain in Subsections~\ref{sec.square} and~\ref{sec.Lshaped}.

\smallskip

\noindent Let the errors in $L^2(\O)$, $H^1(\O)$, and $H^2(\O)$ norms be denoted by
\begin{align*}
\err(u):=\|u-\Pi^hu_h\|_{0,h},\, \err(\nabla u):=|u-\Pi^hu_h|_{1,h},\, \mbox{ and } \err(Hu)=|u-\Pi^hu_h|_{2,h},
\end{align*}
where $\Pi^h$ is the elliptic projection operator in \eqref{eqn.Pih}. The model problem is constructed in such a way that the exact solution is known.
		\begin{figure}[h!!]
	\begin{center}
		\begin{minipage}[b]{0.45\linewidth}
			{\includegraphics[width=8cm]{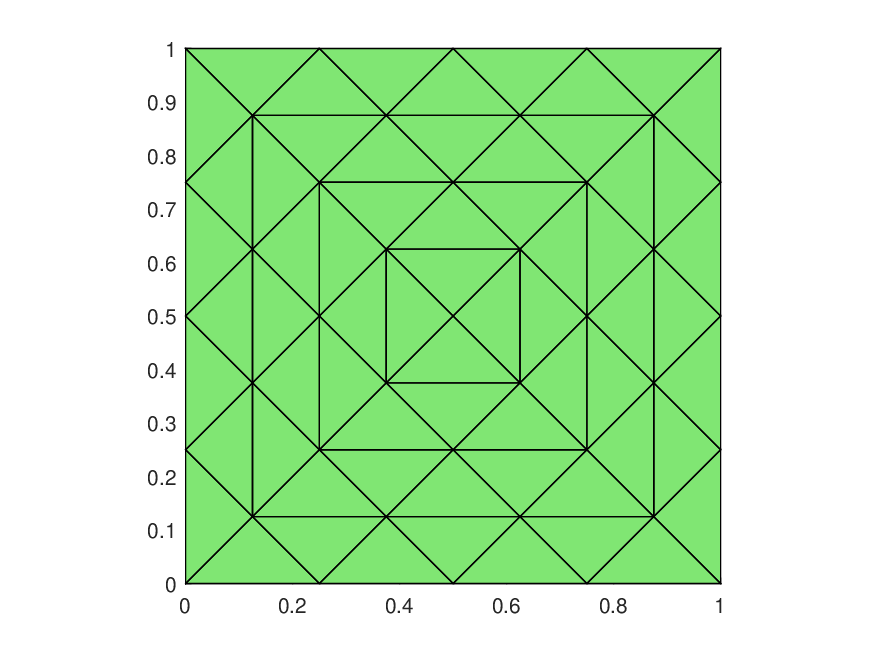}}
			\caption{Triangular Mesh}\label{fig.Triangle}
		\end{minipage}
		\begin{minipage}[b]{0.45\linewidth}
		{\includegraphics[width=8cm]{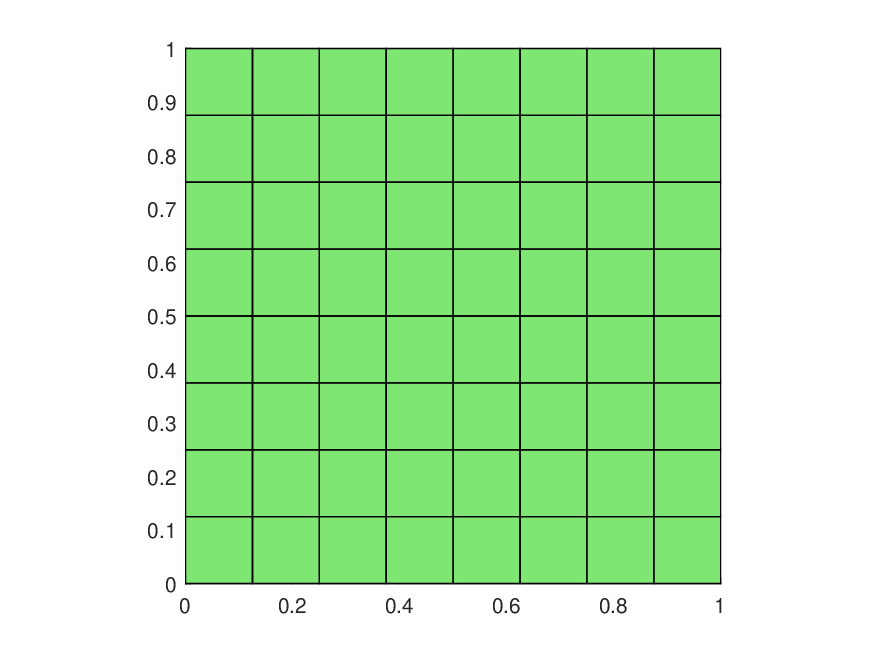}}
		\caption{Square Mesh}\label{fig.Square}
	\end{minipage}
		\begin{minipage}[b]{0.45\linewidth}
			{\includegraphics[width=8cm]{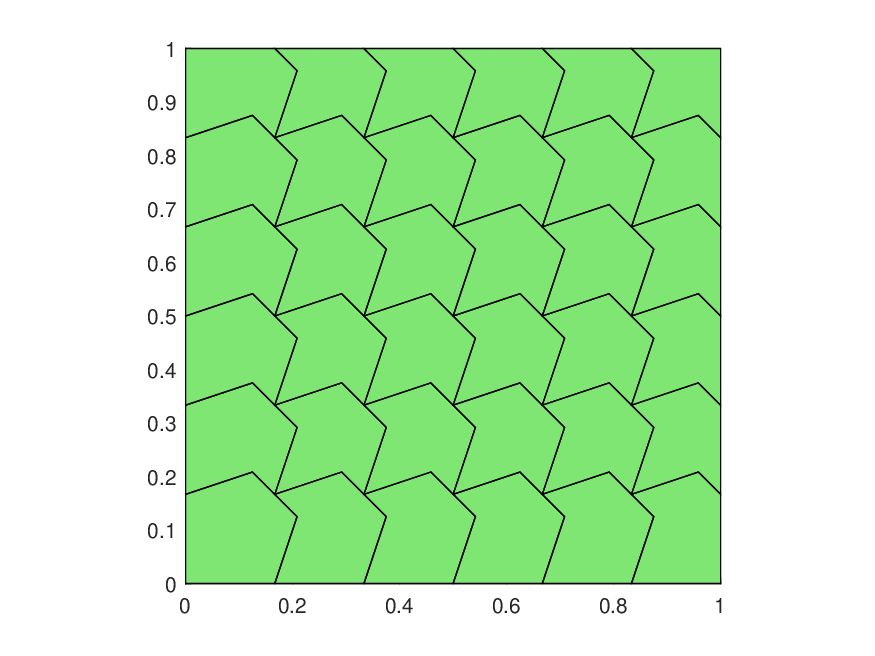}}
				\caption{Concave Mesh}\label{fig.concave}
		\end{minipage}
		\begin{minipage}[b]{0.45\linewidth}
			{\includegraphics[width=8cm]{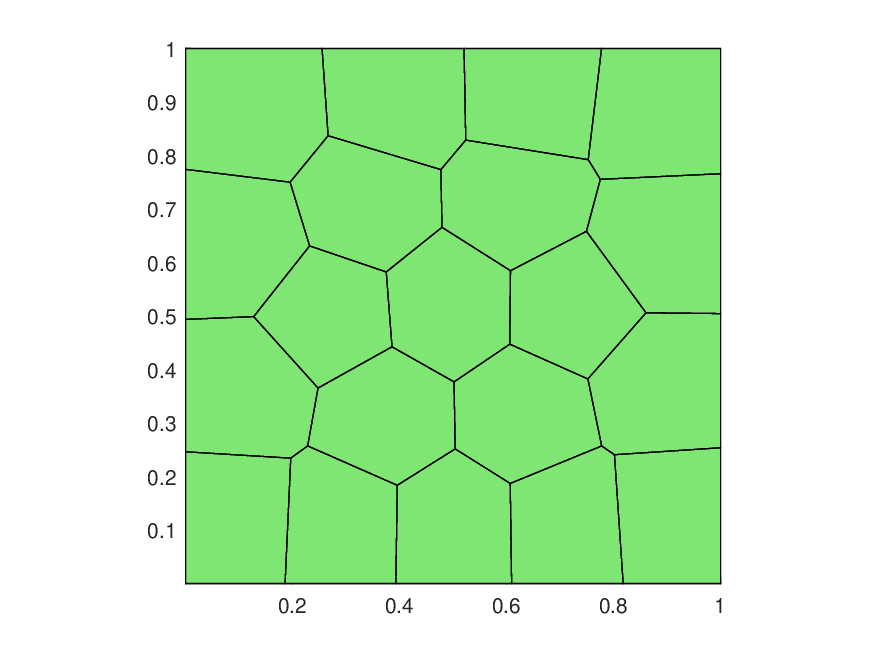}}
				\caption{Structured Voronoi Mesh}\label{fig.SV}
		\end{minipage}
	\begin{minipage}[b]{0.45\linewidth}
	{\includegraphics[width=8cm]{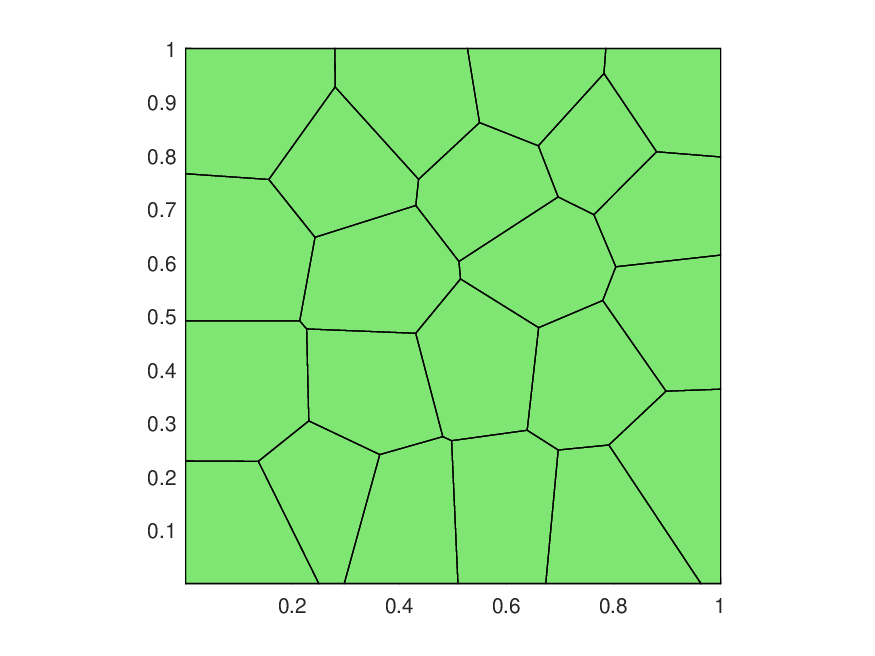}}
		\caption{Random Voronoi Mesh}\label{fig.RV}
\end{minipage}
	\begin{minipage}[b]{0.45\linewidth}
	{\includegraphics[width=8cm]{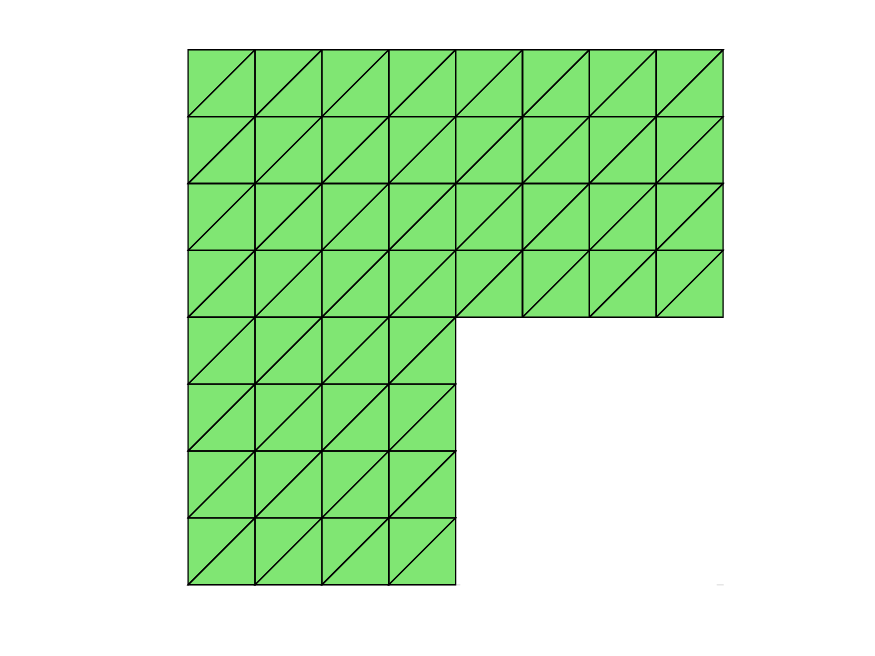}}
	\caption{Triangular Mesh}\label{fig.Lshaped}
\end{minipage}
	\end{center}
\end{figure}

\subsection{Example on the square domain}\label{sec.square}
\noindent Let the computational domain be $\Omega=(0,1)^2$ and the exact solution be given by $u=x^2y^2(1-x)^2(1-y)^2$ and $v=\sin^2(\pi x)\sin^2(\pi y)$.  Then the right hand side load functions are computed as $f=\Delta^2u-[u,v]$ and $g=\Delta^2v+\frac{1}{2}[u,u]$. A series of triangular, square, concave, structured Voronoi, and random Voronoi meshes (see Figure~\ref{fig.Triangle}-\ref{fig.RV}) are employed to test the convergence results for the VEM. We observe in this example that the Newtons’ method converges in three iterations with a tolerance level of $10^{-8}$. 

\smallskip

\noindent Table~\ref{table.eg2.Triangle}-\ref{table.eg2.RV} show  errors and orders of convergence for the displacement $u$ and the Airy-stress function $v$ for the aforementioned five types of meshes. Observe that linear order of convergences are obtained for $u$ and $v$ in the energy norm, and quadratic order of convergence in $L^2$ and $H^1$ norms, see Tables~\ref{table.eg2.Triangle}-\ref{table.eg2.SV}. These numerical order of convergence clearly matches the expected order of convergence given in \eqref{eqn.energyerror.modified}, Theorems~\ref{thm:lowererror}, and~\ref{thm:lowererrorl2}.
\begin{table}[h!!]
	\caption{\small{Convergence results, Square domain, Triangular mesh}}
	{\small{\footnotesize					\begin{center}
				\begin{tabular}{|c ||c|c||c | c ||c|c|}
					\hline
					$h$&$\err(u)$ & Order  & $\err(\nabla u)$ & Order  &$\err(Hu)$ & Order  \\ 
					\hline
					0.500000&0.003848&-&0.010163&-&0.087728&-\\
					0.250000&0.000919&2.0659&0.002565&1.9864&0.040578&1.1123\\
					0.125000&0.000248&1.8892&0.000730&1.8124&0.020991&0.9509\\
					0.062500&0.000064&1.9633&0.000191&1.9382&0.010621&0.9829\\
					0.031250&0.000016&1.9900&0.000048&1.9831&0.005328&0.9954			
					\\	\hline				
				\end{tabular}
				\begin{tabular}{ |c||c |c||c| c|| c|c|}
					\hline
					$h$&$\err(v)$ & Order&$\err(\nabla v)$ & Order  &$\err(Hv)$ & Order  \\ 
					\hline
					0.500000&0.767727&-&2.122015&-&19.371564&-\\
					0.250000&0.177680&2.1113&0.567581&1.9025&9.503684&1.0274\\
					0.125000&0.048263&1.8803&0.161082&1.8170&5.054876&0.9108\\
					0.062500&0.012392&1.9615&0.041987&1.9398&2.575889&0.9726\\
					0.031250&0.003121&1.9895&0.010620&1.9832&1.294492& 0.9927 \\							
					\hline				
				\end{tabular}
		\end{center}}					
	}\label{table.eg2.Triangle}	
\end{table}
\begin{table}[h!!]
	\caption{\small{Convergence results, Square domain, Square mesh}}
	{\small{\footnotesize					\begin{center}
				\begin{tabular}{|c ||c
						|c||c | c ||c|c|}
					\hline
					$h$&$\err(u)$ & Order  & $\err(\nabla u)$ & Order  &$\err(Hu)$ & Order  \\ 
					\hline
					
					0.176777&0.002259&-&0.004788&-&0.037658&-\\	
					0.088388&0.000728&1.6342&0.001567&1.6117&0.014181&1.4089\\
					0.044194&0.000198&1.8802&0.000439&1.8353&0.005473&1.3736\\
					0.022097&0.000051&1.9632&0.000115&1.9388&0.002382&1.1200	\\
					0.011049&0.000013&1.9899&0.000029&1.9818&0.001134&1.0706
					\\	\hline				
				\end{tabular}
				\begin{tabular}{ |c||c |c||c| c|| c|c|}
					\hline
					$h$&$\err(v)$ & Order&$\err(\nabla v)$ & Order  &$\err(Hv)$ & Order  \\ 
					\hline
			    	0.176777&0.340448&-&0.647508&-&7.335924&-\\
					0.088388&0.109373&1.6382&0.194961&1.7317&2.871195&1.3533\\
					0.044194&0.029214&1.9045&0.050934&1.9365&1.130682&1.3445\\
					0.022097&0.007421&1.9769&0.012843&1.9876&0.507089&1.1569 \\			
				    0.011049&0.001862&1.9944&0.003216&1.9975&0.245155&1.0485 \\	\hline				
				\end{tabular}
		\end{center}}					
	}\label{table.eg2.Square}	
\end{table}
\begin{table}[h!!]
	\caption{\small{Convergence results, Square domain, Concave mesh}}
	{\small{\footnotesize					\begin{center}
				\begin{tabular}{|c ||c|c||c | c ||c|c|}
					\hline
					$h$&$\err(u)$ & Order  & $\err(\nabla u)$ & Order  &$\err(Hu)$ & Order  \\ 
					\hline
0.242956&0.006516&-&0.011816&-&0.043874&-\\
0.121478&0.001881&1.7929&0.003695&1.6772&0.020415&1.1037\\
0.060739&0.000520&1.8534&0.001120&1.7215&0.008999&1.1818\\
0.030370&0.000138&1.9118&0.000319&1.8141&0.003882&1.2130\\
0.015185&0.000036&1.9570&0.000085&1.9089&0.001727&1.1686	
					\\	\hline				
				\end{tabular}
				\begin{tabular}{ |c||c |c||c| c|| c|c|}
					\hline
					$h$&$\err(v)$ & Order&$\err(\nabla v)$ & Order  &$\err(Hv)$ & Order  \\ 
					\hline
0.242956&0.985669&-&2.021104&-&7.247490&-\\
0.121478&0.313766&1.6514&0.644574&1.6487&3.459449&1.0669\\
0.060739&0.086675&1.8560&0.179839&1.8416&1.548383&1.1598\\
0.030370&0.022635&1.9370&0.047895&1.9088&0.721287&1.1021\\		
0.015185&0.005770&1.9720&0.012390&1.9507&0.346804&1.0565 \\					
					\hline				
				\end{tabular}
		\end{center}}					
	}\label{table.eg2.concave}	
\end{table}
\begin{table}[h!!]
	\caption{\small{Convergence results, Square domain, Structured Voronoi Mesh}}
	{\small{\footnotesize					\begin{center}
				\begin{tabular}{|c ||c
						|c||c | c ||c|c|}
					\hline
					$h$&$\err(u)$ & Order  & $\err(\nabla u)$ & Order  &$\err(Hu)$ & Order  \\ 
					\hline
0.340697&0.005826&-&0.011623&-&0.071067&-\\
0.171923&0.002225&1.4075&0.004481&1.3935&0.030928&1.2164\\
0.083555&0.000650&1.7061&0.001476&1.5390&0.013413&1.1578\\
0.047445&0.000213&1.9673&0.000537&1.7855&0.006463&1.2900	\\
0.027786&0.000074&1.9693&0.000195&1.8969&0.003328&1.2404\\
						\hline				
				\end{tabular}
				\begin{tabular}{ |c||c |c||c| c|| c|c|}
					\hline
					$h$&$\err(v)$ & Order&$\err(\nabla v)$ & Order  &$\err(Hv)$ & Order  \\ 
					\hline
0.340697&0.906489&-&2.598308&-&10.88671&-\\
0.171923&0.370324&1.3089&0.935058&1.4943&4.857832&1.1798\\
0.083555&0.110610&1.6747&0.267744&1.7332&2.300755&1.0358\\
0.047445&0.035905&1.9881&0.091163&1.9037&1.187114& 1.1692\\			
0.027786&0.012727&1.9384&0.033555&1.8680&0.654228&1.1136\\	
					\hline				
				\end{tabular}
		\end{center}}					
	}\label{table.eg2.SV}	
\end{table}
\begin{table}[h!!]
	\caption{\small{Convergence results, Square domain, Random Voronoi Mesh}}
	{\small{\footnotesize					\begin{center}
				\begin{tabular}{|c ||c
						|c||c | c ||c|c|}
					\hline
					$h$&$\err(u)$ & Order  & $\err(\nabla u)$ & Order  &$\err(Hu)$ & Order  \\ 
					\hline
	0.373676&0.006044&-&0.012864&-&0.072629&-\\
	0.174941&0.002051&1.4240&0.004305&1.4422&0.028624&1.2269\\
	0.089478&0.000502&2.1000&0.001204&1.9008&0.011996&1.2972\\
	0.041643&0.000109&1.9973&0.000285&1.8855&0.004432&1.3017	\\	
	0.020068&0.000035&1.5357&0.000095&1.4965&0.002194&0.9631\\
	\hline				
				\end{tabular}
				\begin{tabular}{ |c||c |c||c| c|| c|c|}
					\hline
					$h$&$\err(v)$ & Order&$\err(\nabla v)$ & Order  &$\err(Hv)$ & Order  \\ 
					\hline
0.373676&0.871569&-&2.513214&--&10.518042&-\\
0.174941&0.350708&1.1995&0.912492&1.3349&4.654239&1.0743\\
0.089478&0.086986&2.0795&0.223585&2.0977&2.118548&1.1739\\
0.041643&0.018494&2.0243&0.049397&1.9741&0.856600& 1.1839\\	
0.020068&0.006145&1.5093&0.016285&1.5201&0.447808&0.8885 \\				
					\hline				
				\end{tabular}
		\end{center}}					
	}\label{table.eg2.RV}	
\end{table}

\subsection{Example on the L-shaped domain}\label{sec.Lshaped}
\noindent Consider the L-shaped domain $\Omega=(-1,1)^2 \setminus\big{(}[0,1)\times(-1,0]\big{)}$. Choose the right hand functions such that the exact singular solution \cite{Grisvard} in polar coordinates is given by
\begin{align*}
u=v=(r^2 \cos^2\theta-1)^2 (r^2 \sin^2\theta-1)^2 r^{1+ \alpha}g_{\alpha,\omega}(\theta),
\end{align*}
where $ \alpha\approx 0.5444837367$ is a non-characteristic 
root of $\sin^2( \alpha\omega) =  \alpha^2\sin^2(\omega)$, $\omega=\frac{3\pi}{2}$,  and
$g_{\alpha,\omega}(\theta)=(\frac{1}{\alpha-1}\sin ((\alpha-1)\omega)-\frac{1}{ \alpha+1}\sin(( \alpha+1)\omega))(\cos(( \alpha-1)\theta)-\cos(( \alpha+1)\theta))$ 
$-(\frac{1}{\alpha-1}\sin(( \alpha-1)\theta)-\frac{1}{ \alpha+1}\sin(( \alpha+1)\theta))
(\cos(( \alpha-1)\omega)-\cos(( \alpha+1)\omega)).$ 
The computation of the discrete solution is executed utilizing triangular meshes, as depicted in Figure~\ref{fig.Lshaped}. In this example, the Newtons' method exhibits convergence within four iterations, while maintaining a tolerance threshold of $10^{-8}$. 

\smallskip

\noindent  This example is particularly interesting since the solution is less regular due to the corner singularity. Since $\O$ is non-convex, we expect only sub-optimal order of convergences in the energy, $H^1$ and $L^2$ norms. Table~\ref{table.Lshaped} confirms these estimates numerically. A similar observation for Morley nonconforming FEM is present in \cite[Section 5]{ng2} and \cite[Section 6.2.2]{JDNNDS_ACOM} for a different weak formulation. 
\begin{table}[h!!]
	\caption{\small{Convergence results, L-shaped domain, Triangular mesh}}
	{\small{\footnotesize					\begin{center}
				\begin{tabular}{|c ||c|c||c | c ||c|c|}
					\hline
					$h$&$\err(u)$ & Order  & $\err(\nabla u)$ & Order  &$\err(Hu)$ & Order  \\ 
					\hline
0.707107&1.515408&-&3.748752&-&20.463977&-\\
0.353553&0.474713&1.6746&1.189979&1.6555&11.578325&0.8217\\
0.176777&0.135893&1.8046&0.347020&1.7778&6.166005&0.9090\\
0.088388&0.039031&1.7998&0.101701&1.7707&3.209604&0.9419\\
0.044194&0.011978&1.7042&0.033175&1.6162&1.682817&0.9315\\
0.022097&0.004070&1.5574&0.012555&1.4018&0.905976&0.8933		
					\\	\hline				
				\end{tabular}
				\begin{tabular}{ |c||c |c||c| c|| c|c|}
					\hline
					$h$&$\err(v)$ & Order&$\err(\nabla v)$ & Order  &$\err(Hv)$ & Order  \\ 
					\hline
0.707107&1.035284&-&2.458626&-&15.791769&-\\
0.353553&0.431224&1.2635&1.084192&1.1812&11.359452&0.4753\\
0.176777&0.124537&1.7919&0.319836&1.7612&6.205400&0.8723\\
0.088388&0.035117&1.8263&0.091430&1.8066&3.236073&0.9393\\
0.044194&0.010510&1.7405&0.029067&1.6533&1.696020&0.9321 \\		
0.022097&0.003484&1.5929&0.010896&1.4156&0.912122&0.8949\\			
					\hline				
				\end{tabular}
		\end{center}}					
	}\label{table.Lshaped}	
\end{table}

	\medskip
 
\noindent {\bf{Acknowledgements.}} The first author thanks Indian Institute of Space Science and Technology (IIST) for the financial support towards the research work. The second author thanks the Department of Science and Technology (DST-SERB), India, for supporting this work through the core research grant CRG/2021/002410.

\bibliographystyle{amsplain}
\bibliography{VEMBib}
\end{document}